\documentclass[english]{amsart}
\usepackage[T1]{fontenc}
\usepackage[latin9]{inputenc}
\usepackage{geometry}
\usepackage{float}
\usepackage{mathrsfs}
\usepackage{amstext}
\usepackage{amsthm}
\usepackage{amssymb}
\usepackage{graphicx}

\makeatletter
\numberwithin{equation}{section}
\numberwithin{figure}{section}
\theoremstyle{plain}
\newtheorem{thm}{\protect\theoremname}[section]
  \theoremstyle{definition}
  \newtheorem{defn}[thm]{\protect\definitionname}
  \theoremstyle{plain}
  \newtheorem{prop}[thm]{\protect\propositionname}
  \theoremstyle{plain}
  \newtheorem{lem}[thm]{\protect\lemmaname}

\makeatother

\usepackage{babel}
  \providecommand{\definitionname}{Definition}
  \providecommand{\lemmaname}{Lemma}
  \providecommand{\propositionname}{Proposition}
\providecommand{\theoremname}{Theorem}

\begin{document}

\title{Asymptotic Properties of Random Voronoi Cells with Arbitrary Underlying
Density}

\author{Isaac Gibbs and Linan Chen}
\begin{abstract}
We consider the Voronoi diagram generated by $n$ i.i.d. $\mathbb{R}^{d}$-valued
random variables with an arbitrary underlying probability density
function $f$ on $\mathbb{R}^{d}$, and analyze the asymptotic behaviours
of certain geometric properties, such as the measure, of the Voronoi
cells as $n$ tends to infinity. We adapt the methods used by Devroye
\emph{et al.} (2017) to conduct a study of the asymptotic properties
of two types of Voronoi cells: 1, Voronoi cells that have a fixed
nucleus; 2, Voronoi cells that contain a fixed point. For the first
type of Voronoi cells, we show that their geometric properties resemble
those in the case when the Voronoi diagram is generated by a homogeneous
Poisson point process. For the second type of Voronoi cells, we determine
the limiting distribution, which is universal in all choices of $f$,
of the re-scaled measure of the cells. For both types, we establish
the asymptotic independence of measures of disjoint cells. 
\end{abstract}

\keywords{point process with arbitrary density, geometry of random Voronoi
diagram, convergence in distribution}

\subjclass[2000]{60G55, 60D05}

\thanks{The first author was supported by the Natural Sciences and Engineering
Research Council (NSERC) of Canada Undergraduate Summer Research Awards
during the process of this work. The second author was partially supported
by the NSERC Discovery Grant (No. 241023).}

\address{Department of Mathematics and Statistics, McGill University, 805
Sherbrooke Street West, Montreal, Quebec, H3A 0B9, Canada.}
\maketitle

\section{Introduction}

\noindent The original motivation of our work is to study the problem
of ascertaining information about an arbitrary probability density
function $f$ over $\mathbb{R}^{d}$ from an independent and identically
distributed (i.i.d.) sample of data whose distribution has $f$ being
the density. This field has a rich history dating back to Rosenblatt's
original proposal of the kernel estimator in \cite{Rosenblatt1956}.
The literature that followed this explores a large variety of potential
solutions and applications. An introduction to some of the main modern
techniques, including histogram and kernel estimators and their applications,
can be found in \cite{White2015}, while alternative approaches based
on combinatorial methods for parameter selection, including results
for both kernel and wavelet estimators, can be found in \cite{Devroye2001}.
Besides, \cite{Devroye2015} provides a contemporary review of nearest
neighbour based techniques.

Rather than focusing on a particular estimator, this work investigates
the relationship between a well studied object, the Voronoi diagram,
and the underlying density. Given $f$ a probability density function
on $\mathbb{R}^{d}$, we will always denote by $\mu_{f}$ the measure
on $\mathbb{R}^{d}$ with $f$ being the density. Let $n\geq1$ be
an integer, and $X_{1},\cdots,X_{n}$ be i.i.d. random variables with
distribution $\mu_{f}$. In the context of this article, we often
refer to $\left\{ X_{1},\cdots,X_{n}\right\} $ as a \emph{point process}.
We denote by $\left\Vert x-y\right\Vert $ the Euclidean distance
between $x$ and $y\in\mathbb{R}^{d}$. For every $i\in\left\{ 1,\cdots,n\right\} $,
we define 
\[
A_{n}\left(X_{i}\right):=\left\{ p\in\mathbb{R}^{d}:\forall j\in\left\{ 1,\cdots,n\right\} \backslash\left\{ i\right\} ,\ ||p-X_{i}||\leq||p-X_{j}||\right\} 
\]
to be the \emph{Voronoi cell} with \emph{nucleus} $X_{i}$ and we
call the collection of cells, $\left\{ A_{n}\left(X_{1}\right),\cdots,A_{n}\left(X_{n}\right)\right\} $,
the \emph{Voronoi diagram} generated by $\left\{ X_{1},\cdots,X_{n}\right\} $.
By noting that each cell is formed by an intersection of half spaces,
one may immediately observe that this defines a partition of $\mathbb{R}^{d}$
into convex polytopes\footnote{Here we also refer to an unbounded convex set formed by an intersection
of half spaces as a ``polytope''. }. Generally speaking, we are interested in studying the asymptotic
behaviour of the Voronoi diagram without assuming any \emph{a priori}
knowledge of the underlying density $f$. Our goal is to establish
properties of the cells as functions (or functionals) of a general
density $f$, or alternatively, to obtain information on $f$ based
solely on the behaviours of the Voronoi diagram.

The applications of Voronoi diagrams span far beyond density estimation
into fields such as astronomy (\cite{Neyrinck2008}), cryptography
(\cite{Pedro2012}), and telecommunication (\cite{Andrews2011}).
More pertinently, these diagrams share a natural link to nearest-neighbour-based
estimation methods, where their study has recently been used to develop
an estimator for the residual variance (\cite{Devroye2018}). For
a more comprehensive overview of the properties of these objects and
their applications we refer the interested reader to \cite{Okabe2000}.\\

\noindent Despite the extensive interest in these structures, previous
work on Voronoi diagrams has largely focused on investigating the
``typical cell'', in the Palm sense (\cite{Moller1994}), in the
case where the sample points arise from a homogeneous Poisson point
process on $\mathbb{R}^{d}$. A number of statistics have been calculated
in this setting, including the first and the second moments of various
geometric quantities, such as the volume, the surface area, and the
number of faces/edges (see, e.g., \cite{Brakke2005a,Brakke2005b,Gilbert1962,Meijering1953,Hayen2002,Henrich1998,Henrich2008}),
and many attempts have been made to estimate the distributions of
these random variables through simulations (see \cite{Vittorietti2017}
and reference therein). A recent article by Devroye \textit{et al.}
(\cite{Devroye2017}) extends this notion of the ``typical'' cell
to the setting presented above, where the diagram is obtained from
$n$ i.i.d. random variables with an arbitrary density function $f$.
More precisely, given a density $f$ on $\mathbb{R}^{d}$ and $n\geq$1,
let $\mu_{f}$ and $\left\{ X_{1},\cdots,X_{n}\right\} $ be the same
as above. We denote by $A_{n}\left(x\right)$ the Voronoi cell with
fixed nucleus $x\in\mathbb{R}^{d}$ in the Voronoi diagram generated
by $\left\{ x,X_{1},\cdots,X_{n}\right\} $ and by $D_{n}^{A}\left(x\right)$
the diameter of $A_{n}\left(x\right)$. Throughout the article, we
restrict our study to the case where $f\left(x\right)>0$ and $x$
is a Lebesgue point of $f$ in the sense that 
\[
\lim_{r\rightarrow0}\frac{1}{\lambda\left(B_{x,r}\right)}\int_{B_{x,r}}\left|f\left(u\right)-f\left(x\right)\right|du=0,
\]
where $\lambda$ is the Lebesgue measure on $\mathbb{R}^{d}$, ``$du$''
is the short form of ``$\lambda\left(du\right)$'', and for every
$z\in\mathbb{R}^{d}$, $r>0$, $B_{z,r}$ is the open ball centered
at $z$ with radius $r$. Then, by the Lebesgue Differentiation Theorem,
$\lambda-$almost every $x\in\mathbb{R}^{d}$ is a Lebesgue point
of $f$, and hence $\mu_{f}-$almost every $x\in\mathbb{R}^{d}$ is
a Lebesgue point of $f$ such that $f\left(x\right)>0$. In this setting,
the authors of \cite{Devroye2017} give a complete characterization
of the limiting distribution, as $n\rightarrow\infty$, of $n\mu_{f}\left(A_{n}\left(x\right)\right)$
by determining the limit of all the moments of $n\mu_{f}\left(A_{n}\left(x\right)\right)$.
Notably, they discover that the limiting distribution of $n\mu_{f}\left(A_{n}\left(x\right)\right)$
is universal in all choices of $f$. In addition, they also show that,
as $n\rightarrow\infty$, $D_{n}^{A}\left(x\right)$ decays probabilistically
at a rate of $n^{-\frac{1}{d}}$.\\

\noindent We hope to extend this work in multiple ways. In Section
2, we revisit the geometric properties of $A_{n}\left(x\right)$ the
cell with fixed nucleus $x$. We discover that when $n$ is large,
$A_{n}\left(x\right)$ can be ``approximately'' viewed as having
arisen from the Voronoi diagram generated by a homogeneous Poisson
point process with the constant parameter $nf\left(x\right)$. Thus,
previous results characterizing the distributions and the moments
of geometric parameters of Voronoi cells generated by Poisson point
processes can be ``transferred'' to the setting with a general underlying
density. In order to state this result more precisely, we now introduce
some basic notations and terminologies. Throughout the article, we
will assume that all the concerned point processes and random variables
are defined on a generic probability space $\left(\Omega,\mathcal{F},\mathbb{P}\right)$.
For every set $A\in\mathcal{F}$, $\mathbb{I}_{A}$ is the indicator
function of $A$. For any point process $\left\{ X_{1},\cdots,X_{n}\right\} $
and any Borel set $B\subseteq\mathbb{R}^{d}$, we define the random
variable
\[
N_{\left\{ X_{1},\cdots,X_{n}\right\} }\left(B\right):=\sum_{i=1}^{n}\mathbb{I}_{\left\{ X_{i}\in B\right\} },
\]
i.e., $N_{\left\{ X_{1},\cdots,X_{n}\right\} }\left(B\right)$ is
the number of points among $\left\{ X_{1},\cdots,X_{n}\right\} $
that fall inside $B$. 
\begin{defn}
\noindent Given $\Lambda>0$, a \emph{homogeneous} \emph{Poisson point
process} with parameter $\Lambda>0$ is a point process $\left\{ Y_{1},\cdots,Y_{n},\cdots\right\} $
with the property that for any positive integer $k$, any bounded
and disjoint sets $B_{1},\cdots,B_{k}\subseteq\mathbb{R}^{d}$, and
any non-negative integers $m_{1},\cdots,m_{k}$, 
\[
\mathbb{P}\left(N_{\left\{ Y_{1},\cdots,Y_{n},\cdots\right\} }\left(B_{i}\right)=m_{i},\;\forall i=1,2,\cdots,k\right)=\prod_{i=1}^{k}\frac{\left(\Lambda\lambda\left(B_{i}\right)\right)^{m_{i}}}{m_{i}!}e^{-\Lambda\lambda(B_{i})}.
\]
\end{defn}

\noindent We will denote by $P_{\Lambda}:=\left\{ Y_{1},\cdots,Y_{n},\cdots\right\} $
such a homogeneous Poisson point process. As we mentioned above, previous
work on the Voronoi diagram generated by $P_{\Lambda}$ has focused
on the study of the ``typical'' or ``average'' cell as defined
by the Palm calculus. This is equivalent to studying the cell with
fixed nucleus $x\in\mathbb{R}^{d}$ in the Voronoi diagram generated
by $\left\{ x\right\} \cup P_{\Lambda}$. Given $x\in\mathbb{R}^{d}$
fixed, assume that $f$ is a probability density function on $\mathbb{R}^{d}$
with $f\left(x\right)>0$. For every $n\geq1$, setting $\Lambda:=nf\left(x\right)$,
we will denote by $P_{n}\left(x\right)$ the cell with nucleus $x$
in the Voronoi diagram generated by the point process $\left\{ x\right\} \cup P_{nf\left(x\right)}$.
With these notations in hand, we can now precisely state the main
result in Section 2.
\begin{thm}
\noindent \label{thm:convergence in distribution main result in Section 2}Let
$f$ be a probability density function on $\mathbb{R}^{d}$ and $x$
be a Lebesgue point of $f$ such that $f(x)>0$. Assume that 
\[
G:\left\{ \text{convex polytopes in }\mathbb{R}^{d}\right\} \rightarrow\mathbb{R}
\]
is any function such that $\forall n\ge1$, $G\left(A_{n}\left(x\right)\right)$
is measurable with respect to $\sigma\left(X_{1},\cdots,X_{n}\right)$,
and $G\left(P_{n}\left(x\right)\right)$ is measurable with respect
to $\sigma\left(P_{nf(x)}\right)$ (e.g. $G(\cdot)$ could be the
number of faces or edges of its input). Then, 
\[
\lim\limits _{n\rightarrow\infty}\sup_{z\in\mathbb{R}}\left|\mathbb{P}\left(G\left(A_{n}\left(x\right)\right)\leq z\right)-\mathbb{P}\left(G\left(P_{n}\left(x\right)\right)\leq z\right)\right|=0.
\]
\end{thm}

\noindent This result provides the link between the Voronoi diagram
generated by the point process with an arbitrary density $f$ and
that generated by a homogeneous Poisson point process, which will
lead to convergence in distribution for a large class of functions
of the Voronoi cell $A_{n}\left(x\right)$. In many cases, we are
also interested in showing that the moments of $G\left(A_{n}\left(x\right)\right)$
are asymptotically close to the moments of $G\left(P_{n}\left(x\right)\right)$.
In order to apply the above result to establishing such properties,
we will require additional controls on $G\left(A_{n}\left(x\right)\right)$
and $G\left(P_{n}\left(x\right)\right)$. For example, in Proposition
\ref{prop:result on bddness of second moment of number of edges}
we examine the case where $d=2$ and $G(\cdot)$ denotes the number
of edges of its inputs, and establish a relationship between $\mathbb{E}\left[G\left(A_{n}\left(x\right)\right)\right]$
and $\mathbb{E}\left[G\left(P_{n}\left(x\right)\right)\right]$ by
controlling the second moment of $G\left(A_{n}\left(x\right)\right)$.
We expect that similar arguments to the ones applied there can be
used to derive the convergence of other moments of interest, as well
as to study other geometric parameters.\\

\noindent In Section 3, we extend the work of \cite{Devroye2017}
to the case where $x$ is not included in the generating point process.
In particular, we focus our study on the cell, $L_{n}\left(x\right)$,
that contains the fixed point $x\in\mathbb{R}^{d}$ in the Voronoi
diagram generated by $\left\{ X_{1},\cdots,X_{n}\right\} $. By applying
similar methods to those in \cite{Devroye2017} to this new case we
are able to obtain both a control on the diameter of $L_{n}\left(x\right)$,
denoted by $D_{n}^{L}\left(x\right)$, and a complete characterization
of the limiting distribution of $n\mu_{f}\left(L_{n}\left(x\right)\right)$
in terms of its limiting moments. More specifically, we have the following
results. 
\begin{thm}
\noindent \label{thm:diameter control on L_n(x), first result in section 3}
Let $f$ be a probability density function on $\mathbb{R}^{d}$ and
$x$ be a Lebesgue point of $f$ such that $f(x)>0$. Then, there
exist universal constants $c_{1},c_{2}>0$ such that $\forall t>0$,
\[
\limsup\limits _{n\rightarrow\infty}\mathbb{P}\left(D_{n}^{L}\left(x\right)\geq\frac{t}{n^{\frac{1}{d}}}\right)\leq c_{1}e^{-c_{2}f\left(x\right)t^{d}}.
\]
\end{thm}

Based on this diameter control, we are able to compute the limit of
all the moments of $n\mu_{f}\left(L_{n}\left(x\right)\right)$ and
hence determine the limiting distribution of $n\mu_{f}\left(L_{n}\left(x\right)\right)$.
\begin{thm}
\label{thm:limiting distribution of n*L_n(x), second result in section 3}For
every positive integer $k$, let $W$ be a Bernoulli$\left(\frac{k}{k+1}\right)$
random variable and $U_{1},\cdots,U_{k}$ be i.i.d. random variables
with the uniform distribution on $B_{0,1}$ that are independent of
$W$. Set $\bar{1}:=\left(1,0,\cdots,0\right)\in\mathbb{R}^{d}$,
and define the random variable 
\[
\begin{split}D_{k}:= & \frac{\lambda\left(B_{U_{1},||\bar{1}-U_{1}||}\cup\dots\cup B_{U_{k},||\bar{1}-U_{k}||}\cup B_{0,1}\right)}{\lambda\left(B_{0,1}\right)}\mathbb{I}_{\left\{ W=0\right\} }\\
 & +\frac{\lambda\left(B_{\bar{1},||\bar{1}-U_{1}||}\cup B_{U_{2},||U_{1}-U_{2}||}\cup B_{U_{3},||U_{1}-U_{3}||}\cup\dots\cup B_{U_{k},||U_{1}-U_{k}||}\cup B_{0,||U_{1}||}\right)}{\lambda\left(B_{0,1}\right)}\mathbb{I}_{\left\{ W=1\right\} }.
\end{split}
\]
Let $f$ be any probability density function on $\mathbb{R}^{d}$
and $x$ be a Lebesgue point of $f$ such that $f\left(x\right)>0$.
Then, 
\[
\lim\limits _{n\rightarrow\infty}\mathbb{E}\left[n^{k}\mu_{f}\left(L_{n}\left(x\right)\right)^{k}\right]=\mathbb{E}\left[\frac{(k+1)!}{D_{k}^{k+1}}\right],\ \forall k\geq1.
\]
Moreover, these limits of the moments uniquely determine a distribution
$\mathscr{D}$ on $\mathbb{R}^{+}$ with the property that $\mathscr{D}$
does not depend on the choice of $f$ or $x$, and the distribution
of $n\mu_{f}\left(L_{n}\left(x\right)\right)$ weakly converges to
$\mathscr{D}$ as $n\rightarrow\infty$. 
\end{thm}

In general, $\mu_{f}\left(L_{n}\left(x\right)\right)$ cannot be computed
without \emph{a priori} knowledge of $f$. Thus, we conclude our study
in Section 3 by investigating the information provided by the Lebesgue
measure of $L_{n}\left(x\right)$. To this end, we prove the following
result. 
\begin{thm}
\label{thm: convergence in distirbution for leb of L_n(x), 3rd result in Section 3}
Let $f$ be any probability density function on $\mathbb{R}^{d}$,
$x$ be a Lebesgue point of $f$ such that $f\left(x\right)>0$, and
$Z$ be a random variable with the distribution $\mathscr{D}$ defined
in Theorem \ref{thm:limiting distribution of n*L_n(x), second result in section 3}.
Then, 
\[
nf\left(x\right)\lambda\left(L_{n}\left(x\right)\right)\rightarrow Z\text{ in distribution}.
\]
\end{thm}

Section 4 is dedicated to showing that for all $n$ sufficiently large,
disjoint regions of the Voronoi diagram behave ``almost'' independently
from one another. In particular, combined with the results from the
previous sections, this gives us a method for studying $f$ in multiple
disjoint regions of $\mathbb{R}^{d}$ simultaneously without requiring
multiple data sets.
\begin{thm}
\noindent \label{thm:asymptotic independence result of Section 4}Let
$k\geq2$ be an integer and $Z_{1},\cdots,Z_{k}$ be i.i.d. random
variables with the distribution $\mathscr{D}$ defined in Theorem
\ref{thm:limiting distribution of n*L_n(x), second result in section 3}.
Assume that $f$ is a probability density function on $\mathbb{R}^{d}$
and $x_{1},\cdots,x_{k}$ are $k$ distinct Lebesgue points of $f$
such that $f\left(x_{1}\right),\cdots,f\left(x_{k}\right)$ are all
positive. Then, 
\[
\left(n\mu_{f}\left(L_{n}\left(x_{1}\right)\right),\cdots,n\mu_{f}\left(L_{n}\left(x_{k}\right)\right)\right)\rightarrow\left(Z_{1},\cdots,Z_{k}\right)\text{ in distribution}.
\]
\end{thm}

\section{Voronoi Cells With Fixed Nucleus}

Assume that $f$, $\mu_{f}$, $x$, $\left\{ X_{1},\cdots,X_{n}\right\} $,
$P_{nf\left(x\right)}$, $A_{n}\left(x\right)$ and $P_{n}\left(x\right)$
are exactly as in the Introduction. We first look to extend the work
of \cite{Devroye2017} to include consideration of other geometric
properties of $A_{n}\left(x\right)$ besides its measure. The way
we achieve this goal is by linking and comparing $A_{n}\left(x\right)$
with $P_{n}\left(x\right)$, for all $n$ sufficiently large. Our
main result of this section is Theorem \ref{thm:convergence in distribution main result in Section 2}
which is stated again below. The theorem shows that for a large class
of functions $G$ on the space of polytopes in $\mathbb{R}^{d}$,
the distributions of $G\left(A_{n}\left(x\right)\right)$ and $G\left(P_{n}\left(x\right)\right)$
become arbitrarily close to each other for $n$ large under the L�vy
metric. The class of such functions $G$ includes most tractable functions
of a polytope, e.g., the volume, the surface area, the number of faces/edges,
the length of each side, the locations of vertices, the location of
the center of the mass, and many more, if not all, geometric properties
that one might be interested in studying. In many cases, we are interested
in obtaining information on the asymptotics of the distributions of
these random variables, e.g., the limit of their moments. With Theorem
\ref{thm:convergence in distribution main result in Section 2} and
the existing results on the Voronoi diagrams generated by homogeneous
Poisson point processes in hand, this task reduces to establishing
proper integrability of the concerned random variable for both $A_{n}\left(x\right)$
and $P_{n}\left(x\right)$. One demonstration of how this can be done
is given in Proposition 2.1 where the expected number of edges of
$A_{n}\left(x\right)$ and $P_{n}\left(x\right)$ is considered in
2D. We leave the consideration of higher moments of this random variable
to future work.\\

\noindent \textbf{Theorem \ref{thm:convergence in distribution main result in Section 2}.
}L\emph{et $f$ be a probability density function on $\mathbb{R}^{d}$
and $x$ be a Lebesgue point of $f$ such that $f\left(x\right)>0$.
Assume that 
\[
G:\left\{ \text{convex polytopes in }\mathbb{R}^{d}\right\} \rightarrow\mathbb{R}
\]
is any function such that $\forall n\ge1$, $G\left(A_{n}\left(x\right)\right)$
is measurable with respect to $\sigma\left(X_{1},\dots,X_{n}\right)$,
and $G\left(P_{n}\left(x\right)\right)$ is measurable with respect
to $\sigma\left(P_{nf\left(x\right)}\right)$. Then, 
\[
\lim\limits _{n\rightarrow\infty}\sup_{z\in\mathbb{R}}\left|\mathbb{P}\left(G\left(A_{n}\left(x\right)\right)\leq z\right)-\mathbb{P}\left(G\left(P_{n}\left(x\right)\right)\leq z\right)\right|=0.
\]
}
\begin{proof}
\noindent Let $f$ and $x$ be as in the statement of the theorem.
The main idea of the proof is that for large $n$, $A_{n}\left(x\right)$
is contained in a small region around $x$ on which $f$ is well approximated
by $f\left(x\right)$. Moreover, by choosing this region appropriately
we can establish that for some constant $c>0$ the number of points
among $X_{1},\cdots,X_{n}$ that fall into this region will have the
distribution Bin($n$,$\frac{c}{n}$), i.e., the binomial distribution
with parameters $n$ and $\frac{c}{n}$, which is ``approximately''
Poisson($c$), the Poisson distribution with parameter $c$, when
$n$ is large. As a result, for all $n$ sufficiently large, $A_{n}\left(x\right)$
can be approximately viewed as having arisen form a homogeneous Poisson
point process. Let us now make this precise.

We first apply the estimates of the diameter obtained in \cite{Devroye2017}
(see Lemma \ref{lem:Appendix control on diameter of A_n(x) and P_n(x)}
in the Appendix), which implies that for arbitrary $\epsilon>0$,
there exists $t>0$ such that for all $n$ sufficiently large,
\[
\mathbb{P}\left(D_{n}^{A}\left(x\right)>\frac{t}{n^{\frac{1}{d}}}\right)<\epsilon\text{ and }\mathbb{P}\left(D_{n}^{P}\left(x\right)>\frac{t}{n^{\frac{1}{d}}}\right)<\epsilon.
\]
Now, with $t$ chosen and fixed, we apply the Lebesgue Differentiation
Theorem (see, e.g., Theorem 20.19 of \cite{Devroye2015}) to get that
for any $\phi>0$ and for all $n$ sufficiently large, 
\[
\left|\frac{\mu_{f}\left(B_{x,2\frac{t}{n^{\frac{1}{d}}}}\right)}{\lambda\left(B_{x,2\frac{t}{n^{\frac{1}{d}}}}\right)}-f\left(x\right)\right|\leq\phi f\left(x\right),
\]
which leads to
\begin{equation}
\frac{\left(1-\phi\right)\lambda\left(B_{0,1}\right)f\left(x\right)2^{d}t^{d}}{n}\leq\mu_{f}\left(B_{x,2\frac{t}{n^{\frac{1}{d}}}}\right)\leq\frac{\left(1+\phi\right)\lambda\left(B_{0,1}\right)f\left(x\right)2^{d}t^{d}}{n}.\label{eq: estimate of mu measure of the ball with radius 2tn^(-1/d)}
\end{equation}
Recall that $N_{\left\{ X_{1},\cdots,X_{n}\right\} }\left(B\right)$
is the number of points among $X_{1},\cdots,X_{n}$ that lie in a
set $B\subseteq\mathbb{R}^{d}$. Then, choose $M>0$ large such that
for all $n$ sufficiently large, 
\[
\begin{split} & \mathbb{P}\left(N_{\left\{ X_{1},\cdots,X_{n}\right\} }\left(B_{x,2\frac{t}{n^{\frac{1}{d}}}}\right)>M\right)\\
= & \text{Bin}\left(n,\mu_{f}\left(B_{x,2t\cdot n^{-\frac{1}{d}}}\right)\right)\left(\left(M,\infty\right)\right)\\
\leq & e^{M-\left(1-\phi\right)\lambda\left(B_{0,1}\right)f\left(x\right)2^{d}t^{d}-M\ln\left(\frac{M}{\left(1+\phi\right)\lambda\left(B_{0,1}\right)f\left(x\right)2^{d}t^{d}}\right)},
\end{split}
\]
where the last inequality is due to (\ref{eq: estimate of mu measure of the ball with radius 2tn^(-1/d)})
and an application of Chernoff's bound (see Lemma \ref{lem: Appendix Chernoff's-bound}
in the Appendix). Clearly, by choosing $M$ large, we can make the
expression above smaller than $\epsilon$. From the above, we obtain
that 
\begin{equation}
\begin{split} & \mathbb{P}\left(G\left(A_{n}\left(x\right)\right)\leq z\right)-\mathbb{P}\left(G\left(A_{n}\left(x\right)\right)\leq z,D_{n}^{A}\left(x\right)\leq\frac{t}{n^{\frac{1}{d}}},N_{\left\{ X_{1},\cdots,X_{n}\right\} }\left(B_{x,2\frac{t}{n^{\frac{1}{d}}}}\right)\leq M\right)\\
\leq & \mathbb{P}\left(D_{n}^{A}\left(x\right)>\frac{t}{n^{\frac{1}{d}}}\right)+\mathbb{P}\left(N_{\left\{ X_{1},\cdots,X_{n}\right\} }\left(B_{x,2\frac{t}{n^{\frac{1}{d}}}}\right)>M\right)\leq2\epsilon.
\end{split}
\label{eq:imposing further constraints on N and diameter for A_n(x)}
\end{equation}
Now, observe that, 
\begin{equation}
\begin{split} & \mathbb{P}\left(G\left(A_{n}\left(x\right)\right)\leq z,D_{n}^{A}(x)\leq\frac{t}{n^{\frac{1}{d}}},N_{\left\{ X_{1},\cdots,X_{n}\right\} }\left(B_{x,2\frac{t}{n^{\frac{1}{d}}}}\right)\leq M\right)\\
= & \sum_{i=0}^{M}\mathbb{P}\left(G\left(A_{n}\left(x\right)\right)\leq z,D_{n}^{A}(x)\leq\frac{t}{n^{\frac{1}{d}}},N_{\left\{ X_{1},\cdots,X_{n}\right\} }\left(B_{x,2\frac{t}{n^{\frac{1}{d}}}}\right)=i\right)\\
= & \sum_{i=0}^{M}\mathbb{P}\left(\left.G\left(A_{n}\left(x\right)\right)\leq z,D_{n}^{A}(x)\leq\frac{t}{n^{\frac{1}{d}}}\right|\left\{ X_{1},\cdots,X_{n}\right\} \cap B_{x,2\frac{t}{n^{\frac{1}{d}}}}=\left\{ X_{1},\cdots,X_{i}\right\} \right)\\
 & \hspace{6cm}\cdot\binom{n}{i}\left[\mu_{f}\left(B_{x,2\frac{t}{n^{\frac{1}{d}}}}\right)\right]^{i}\left[1-\mu_{f}\left(B_{x,2\frac{t}{n^{\frac{1}{d}}}}\right)\right]^{n-i}.
\end{split}
\label{eq: computation of the prob on A_n(x) with further constraints}
\end{equation}

On the other hand, for every $n\geq1$, in the case when the Voronoi
diagram is generated by the homogeneous Poisson point process $P_{nf\left(x\right)}$,
$N_{P_{nf\left(x\right)}}\left(B_{x,2\frac{t}{n^{\frac{1}{d}}}}\right)$
is just a Poisson random variable with parameter 
\[
nf\left(x\right)\lambda\left(B_{x,2\frac{t}{n^{\frac{1}{d}}}}\right)=\lambda\left(B_{0,1}\right)f\left(x\right)2^{d}t^{d}
\]
which does not depend on $n$. Thus, by choosing $M$ sufficiently
large, we can make 
\[
\mathbb{P}\left(N_{P_{nf\left(x\right)}}\left(B_{x,2\frac{t}{n^{\frac{1}{d}}}}\right)>M\right)<\epsilon
\]
for all $n\geq1$, and hence we also have 
\begin{equation}
\mathbb{P}\left(G\left(P_{n}\left(x\right)\right)\leq z\right)-\mathbb{P}\left(G\left(P_{n}\left(x\right)\right)\leq z,D_{n}^{P}\left(x\right)\leq\frac{t}{n^{\frac{1}{d}}},N_{P_{nf\left(x\right)}}\left(B_{x,2\frac{t}{n^{\frac{1}{d}}}}\right)\leq M\right)\leq2\epsilon.\label{eq:imposing further constraints on N and diameter for P_n(x)}
\end{equation}
Further, 
\begin{equation}
\begin{split} & \mathbb{P}\left(G\left(P_{n}\left(x\right)\right)\leq z,D_{n}^{P}\left(x\right)\leq\frac{t}{n^{\frac{1}{d}}},N_{P_{nf\left(x\right)}}\left(B_{x,2\frac{t}{n^{\frac{1}{d}}}}\right)\leq M\right)\\
= & \sum_{i=0}^{M}\mathbb{P}\left(G\left(P_{n}\left(x\right)\right)\leq z,D_{n}^{P}\left(x\right)\leq\frac{t}{n^{\frac{1}{d}}},N_{P_{nf\left(x\right)}}\left(B_{x,2\frac{t}{n^{\frac{1}{d}}}}\right)=i\right)\\
= & \sum_{i=0}^{M}\mathbb{P}\left(\left.G\left(P_{n}\left(x\right)\right)\leq z,D_{n}^{P}\left(x\right)\leq\frac{t}{n^{\frac{1}{d}}}\right|N_{P_{nf\left(x\right)}}\left(B_{x,2\frac{t}{n^{\frac{1}{d}}}}\right)=i\right)\\
 & \hspace{6cm}\cdot\frac{\left[\lambda\left(B_{0,1}\right)f\left(x\right)2^{d}t^{d}\right]^{i}}{i!}e^{-\lambda\left(B_{0,1}\right)f\left(x\right)2^{d}t^{d}}.
\end{split}
\label{eq:computation of the prob on P_n(x) with further constraints}
\end{equation}

Combining (\ref{eq:imposing further constraints on N and diameter for A_n(x)})
- (\ref{eq:computation of the prob on P_n(x) with further constraints}),
we conclude that in order to control 
\[
\sup_{z\in\mathbb{R}}\left|\mathbb{P}\left(G\left(A_{n}\left(x\right)\right)\leq z\right)-\mathbb{P}\left(G\left(P_{n}\left(x\right)\right)\leq z\right)\right|,
\]

\noindent it is sufficient to prove the following two facts.\\

\noindent \textbf{Fact 1. }For every $i\in\left\{ 0,\cdots,M\right\} $,
\[
\lim\limits _{n\rightarrow\infty}\binom{n}{i}\left[\mu_{f}\left(B_{x,2\frac{t}{n^{\frac{1}{d}}}}\right)\right]^{i}\left[1-\mu_{f}\left(B_{x,2\frac{t}{n^{\frac{1}{d}}}}\right)\right]^{n-i}=\frac{\left[\lambda\left(B_{0,1}\right)f\left(x\right)2^{d}t^{d}\right]^{i}}{i!}e^{-\lambda\left(B_{0,1}\right)f\left(x\right)2^{d}t^{d}}.
\]
\textbf{Fact 2. }For every $i\in\left\{ 0,\cdots,M\right\} $, 
\[
\begin{split} & \lim\limits _{n\rightarrow\infty}\sup_{z\in\mathbb{R}}\left|\mathbb{P}\left(\left.G\left(A_{n}\left(x\right)\right)\leq z,\ D_{n}^{A}\left(x\right)\leq\frac{t}{n^{\frac{1}{d}}}\right|\left\{ X_{1},\cdots,X_{n}\right\} \cap B_{x,2\frac{t}{n^{\frac{1}{d}}}}=\left\{ X_{1},\cdots,X_{i}\right\} \right)\right.\\
 & \hspace{4.5cm}-\left.\mathbb{P}\left(\left.G\left(P_{n}\left(x\right)\right)\leq z,\ D_{n}^{P}\left(x\right)\leq\frac{t}{n^{\frac{1}{d}}}\right|N_{P_{nf\left(x\right)}}\left(B_{x,2\frac{t}{n^{\frac{1}{d}}}}\right)=i\right)\right|=0.
\end{split}
\]

\noindent \emph{Proof of Fact 1:} Fact 1 follows from a straightforward
application of the Lebesgue Differentiation Theorem. To be specific,
let $\phi>0$ be arbitrary. Due to (\ref{eq: estimate of mu measure of the ball with radius 2tn^(-1/d)}),
we have that 
\[
\begin{split} & \limsup_{n\rightarrow\infty}\binom{n}{i}\left[\mu_{f}\left(B_{x,2\frac{t}{n^{\frac{1}{d}}}}\right)\right]^{i}\left[1-\mu_{f}\left(B_{x,2\frac{t}{n^{\frac{1}{d}}}}\right)\right]^{n-i}\\
\leq & \limsup_{n\rightarrow\infty}{n \choose i}\left[\frac{\left(1+\phi\right)\lambda\left(B_{0,1}\right)f\left(x\right)2^{d}t^{d}}{n}\right]^{i}\left[1-\frac{\left(1-\phi\right)\lambda\left(B_{0,1}\right)f\left(x\right)2^{d}t^{d}}{n}\right]^{n-i}\\
= & \frac{\left[\left(1+\phi\right)\lambda\left(B_{0,1}\right)f\left(x\right)2^{d}t^{d}\right]^{i}}{i!}e^{-\left(1-\phi\right)\lambda\left(B_{0,1}\right)f\left(x\right)2^{d}t^{d}}
\end{split}
\]
and similarly, 
\[
\liminf_{n\rightarrow\infty}\binom{n}{i}\left[\mu_{f}\left(B_{x,2\frac{t}{n^{\frac{1}{d}}}}\right)\right]^{i}\left[1-\mu_{f}\left(B_{x,2\frac{t}{n^{\frac{1}{d}}}}\right)\right]^{n-i}\geq\frac{\left[\left(1-\phi\right)\lambda\left(B_{0,1}\right)f\left(x\right)2^{d}t^{d}\right]^{i}}{i!}e^{-\left(1+\phi\right)\lambda\left(B_{0,1}\right)f\left(x\right)2^{d}t^{d}}.
\]
Since $\phi>0$ is arbitrary, we have proven Fact 1.\\

\noindent \emph{Proof of Fact 2:} Let $x$, $t$, $i$ and $n$ be
the same as above. To simplify the notations, for every $z\in\mathbb{R}$,
we will denote by $E_{z}$ the collection of any set of $i$ points
$\{x_{1},\dots,x_{i}\}\subseteq\mathbb{R}^{d}$, such that the following
is true:

\noindent (i) $\left\{ x_{1},\cdots,x_{i}\right\} \subseteq B_{x,2\frac{t}{n^{\frac{1}{d}}}}$;
\\
(ii) if $A_{i}\left(x\right)$ is the cell with nucleus $x$ in the
Voronoi diagram generated by $\left\{ x,x_{1},\cdots,x_{i}\right\} $,
then 
\[
G\left(A_{i}\left(x\right)\right)\leq z;
\]
(iii) if $D_{i}^{A}\left(x\right)$ is the diameter of $A_{i}\left(x\right)$,
then $D_{i}^{A}\left(x\right)\leq\frac{t}{n^{\frac{1}{d}}}$. 

By Lemma \ref{lem:Appendix range of nuclei relevant to configuration of cell }
in the Appendix, we know that whenever $D_{n}^{A}\left(x\right)\leq\frac{t}{n^{\frac{1}{d}}}$
(alternatively $D_{n}^{P}\left(x\right)\leq\frac{t}{n^{\frac{1}{d}}}$),
points that are outside of $B_{x,2\frac{t}{n^{\frac{1}{d}}}}$ cannot
affect the shape of the cell with nucleus $x$, or in other words,
only points that are inside $B_{x,2\frac{t}{n^{\frac{1}{d}}}}$ should
be considered when studying the configuration of $A_{n}\left(x\right)$.
As a consequence, given that $X_{1},\cdots,X_{i}$ are the only points
among $X_{1},\cdots,X_{n}$ that are inside $B_{x,2\frac{t}{n^{\frac{1}{d}}}}$,
$A_{n}\left(x\right)$ is the same as $A_{i}\left(x\right)$ when
only the points $\left\{ x,X_{1},\cdots,X_{i}\right\} $ are used
in generating the Voronoi diagram. Therefore, for any $z\in\mathbb{R}$,
\[
\begin{split} & \mathbb{P}\left(\left.G\left(A_{n}\left(x\right)\right)\leq z,D_{n}^{A}\left(x\right)\leq\frac{t}{n^{\frac{1}{d}}}\right|\left\{ X_{1},\cdots,X_{n}\right\} \cap B_{x,2\frac{t}{n^{\frac{1}{d}}}}=\left\{ X_{1},\cdots,X_{i}\right\} \right)\\
= & \frac{\mathbb{P}\left(\left\{ X_{1},\dots,X_{i}\right\} \in E_{z},X_{i+1},\dots,X_{n}\notin B_{x,2\frac{t}{n^{\frac{1}{d}}}}\right)}{\mathbb{P}\left(\left\{ X_{1},\cdots,X_{n}\right\} \cap B_{x,2\frac{t}{n^{\frac{1}{d}}}}=\left\{ X_{1},\cdots,X_{i}\right\} \right)}\\
= & \frac{\mathbb{P}\left(\left\{ X_{1},\dots,X_{i}\right\} \in E_{z}\right)\mathbb{P}\left(X_{i+1},\dots,X_{n}\notin B_{x,2\frac{t}{n^{\frac{1}{d}}}}\right)}{\mathbb{P}\left(X_{1},\dots,X_{i}\in B_{x,2\frac{t}{n^{\frac{1}{d}}}}\right)\mathbb{P}\left(X_{i+1},\dots,X_{n}\notin B_{x,2\frac{t}{n^{\frac{1}{d}}}}\right)}\\
= & \mathbb{P}\left(\left\{ X_{1},\dots,X_{i}\right\} \in E_{z}\left|X_{1},\dots,X_{i}\in B_{x,2\frac{t}{n^{\frac{1}{d}}}}\right.\right).
\end{split}
\]
By a similar argument as above and basic properties of homogeneous
Poisson point processes, we also have 
\[
\mathbb{P}\left(\left.G\left(P_{n}\left(x\right)\right)\leq z,D_{n}^{P}\left(x\right)\leq\frac{t}{n^{\frac{1}{d}}}\right|N_{P_{nf\left(x\right)}}\left(B_{x,2\frac{t}{n^{\frac{1}{d}}}}\right)=i\right)=\mathbb{P}\left(\left\{ U_{1},\cdots,U_{i}\right\} \in E_{z}\right)
\]
where $U_{1},\cdots,U_{i}$ are i.i.d. random variables with the uniform
distribution on $B_{x,2\frac{t}{n^{\frac{1}{d}}}}$.

It is easy to see that the conditional distribution of $\left\{ X_{1},\cdots,X_{n}\right\} $,
conditioning on the event that $X_{1},\cdots,X_{i}\in B_{x,2\frac{t}{n^{\frac{1}{d}}}}$,
has the probability density function 
\[
\frac{f\left(x_{1}\right)f\left(x_{2}\right)\cdots f\left(x_{i}\right)}{\left[\mu_{f}\left(B_{x,2\frac{t}{n^{\frac{1}{d}}}}\right)\right]^{i}}\text{ for every }x_{1},\cdots,x_{i}\in B_{x,2\frac{t}{n^{\frac{1}{d}}}}.
\]
Thus, to prove Fact 2, it is enough for us to show that 
\[
\begin{split} & \left|\mathbb{P}\left(\left\{ X_{1},\dots,X_{i}\right\} \in E_{z}\left|X_{1},\dots,X_{i}\in B_{x,2\frac{t}{n^{\frac{1}{d}}}}\right.\right)-\mathbb{P}\left(\left\{ U_{1},\cdots,U_{i}\right\} \in E_{z}\right)\right|\\
\leq & \idotsint_{E_{z}}\left|\prod_{j=1}^{i}\frac{f\left(x_{j}\right)}{\mu_{f}\left(B_{x,2\frac{t}{n^{\frac{1}{d}}}}\right)}-\prod_{j=1}^{i}\frac{n}{\lambda\left(B_{0,1}\right)2^{d}t^{d}}\right|dx_{1}\cdots dx_{i}\\
\leq & \int_{B_{x,2\frac{t}{n^{\frac{1}{d}}}}}\cdots\int_{B_{x,2\frac{t}{n^{\frac{1}{d}}}}}\left|\prod_{j=1}^{i}\frac{f\left(x_{j}\right)}{\mu_{f}\left(B_{x,2\frac{t}{n^{\frac{1}{d}}}}\right)}-\prod_{j=1}^{i}\frac{n}{\lambda\left(B_{0,1}\right)2^{d}t^{d}}\right|dx_{1}\cdots dx_{i}\\
\rightarrow & 0\text{ as }n\rightarrow\infty.
\end{split}
\]
Let $\phi\in(0,1)$ be arbitrary. Using Lemma \ref{lem:Appendix Lebesgue density theorem}
in the Appendix, we have that $\exists\delta>0$, such that for every
$\eta\in\left(0,\delta\right)$, 
\[
\int_{B_{x,\eta}}\left|\frac{f\left(u\right)}{\mu_{f}\left(B_{x,\eta}\right)}-\frac{1}{\lambda\left(B_{0,1}\right)\eta^{d}}\right|du\leq\phi.
\]
It follows that for all $n$ sufficiently large, 
\[
\begin{split} & \int_{B_{x,2\frac{t}{n^{\frac{1}{d}}}}}\cdots\int_{B_{x,2\frac{t}{n^{\frac{1}{d}}}}}\left|\prod_{j=1}^{i}\frac{f\left(x_{j}\right)}{\mu_{f}\left(B_{x,2\frac{t}{n^{\frac{1}{d}}}}\right)}-\prod_{j=1}^{i}\frac{n}{\lambda\left(B_{0,1}\right)2^{d}t^{d}}\right|dx_{1}\cdots dx_{i}\\
\leq & \int_{B_{x,2\frac{t}{n^{\frac{1}{d}}}}}\cdots\int_{B_{x,2\frac{t}{n^{\frac{1}{d}}}}}\left|\prod_{j=2}^{i}\frac{f\left(x_{j}\right)}{\mu_{f}\left(B_{x,2\frac{t}{n^{\frac{1}{d}}}}\right)}-\prod_{j=2}^{i}\frac{n}{\lambda\left(B_{0,1}\right)2^{d}t^{d}}\right|\frac{f\left(x_{1}\right)}{\mu_{f}\left(B_{x,2\frac{t}{n^{\frac{1}{d}}}}\right)}dx_{1}\cdots dx_{i}\\
 & \hspace{2cm}+\int_{B_{x,2\frac{t}{n^{\frac{1}{d}}}}}\cdots\int_{B_{x,2\frac{t}{n^{\frac{1}{d}}}}}\left|\frac{f\left(x_{1}\right)}{\mu_{f}\left(B_{x,2\frac{t}{n^{\frac{1}{d}}}}\right)}-\frac{n}{\lambda\left(B_{0,1}\right)2^{d}t^{d}}\right|\prod_{j=2}^{i}\frac{n}{\lambda\left(B_{0,1}\right)2^{d}t^{d}}dx_{1}\cdots dx_{i}\\
\leq & \int_{B_{x,2\frac{t}{n^{\frac{1}{d}}}}}\cdots\int_{B_{x,2\frac{t}{n^{\frac{1}{d}}}}}\left|\prod_{j=2}^{i}\frac{f\left(x_{j}\right)}{\mu_{f}\left(B_{x,2\frac{t}{n^{\frac{1}{d}}}}\right)}-\prod_{j=2}^{i}\frac{n}{\lambda\left(B_{0,1}\right)2^{d}t^{d}}\right|dx_{2}\cdots dx_{i}\\
 & \hspace{6cm}+\phi\int_{B_{x,2\frac{t}{n^{\frac{1}{d}}}}}\cdots\int_{B_{x,2\frac{t}{n^{\frac{1}{d}}}}}\left(\frac{n}{\lambda\left(B_{0,1}\right)2^{d}t^{d}}\right)^{i-1}dx_{2}\cdots dx_{i}\\
= & \int_{B_{x,2\frac{t}{n^{\frac{1}{d}}}}}\cdots\int_{B_{x,2\frac{t}{n^{\frac{1}{d}}}}}\left|\prod_{j=2}^{i}\frac{f\left(x_{j}\right)}{\mu_{f}\left(B_{x,2\frac{t}{n^{\frac{1}{d}}}}\right)}-\prod_{j=2}^{i}\frac{n}{\lambda\left(B_{0,1}\right)2^{d}t^{d}}\right|dx_{2}\cdots dx_{i}+\phi\\
\leq & i\phi\text{ (by repeating this process \ensuremath{i} times)}.
\end{split}
\]
Therefore, we have that 
\[
\lim_{n\rightarrow\infty}\sup_{z\in\mathbb{R}}\left|\mathbb{P}\left(\left\{ X_{1},\dots,X_{i}\right\} \in E_{z}\left|X_{1},\dots,X_{i}\in B_{x,2\frac{t}{n^{\frac{1}{d}}}}\right.\right)-\mathbb{P}\left(\left\{ U_{1},\cdots,U_{i}\right\} \in E_{z}\right)\right|=0,
\]
which concludes the proof of Fact 2, as well as the proof of Theorem
\ref{thm:convergence in distribution main result in Section 2}.
\end{proof}
\noindent We now give an explicit example of how Theorem \ref{thm:convergence in distribution main result in Section 2}
can be used to derive a limit for the expected number of edges of
$A_{n}\left(x\right)$ when $d=2$.
\begin{prop}
\label{prop:result on bddness of second moment of number of edges}Let
$d=2$, $f$ be a probability density function on $\mathbb{R}^{d}$
and $x$ be a Lebesgue point of $f$ such that $f\left(x\right)>0$.
Then, 
\[
\sup_{n\geq1}\mathbb{E}\left[\left(\text{the number of edges of }A_{n}\left(x\right)\right)^{2}\right]<\infty.
\]
Moreover, 
\[
\lim\limits _{n\rightarrow\infty}\mathbb{E}\left[\text{the number of edges of }A_{n}\left(x\right)\right]=6.
\]
\end{prop}

\begin{proof}
\noindent We recall that the expected number of edges of $P_{n}\left(x\right)$
is equal to 6 (\cite{Meijering1953}). Also, by a direct corollary
of Theorem 5.3 of \cite{Henrich2008}, we have that for every $n\geq1$,
\[
\mathbb{E}\left[\left(\text{the number of edges of }P_{n}\left(x\right)\right)^{2}\right]=\mathbb{E}\left[\left(\text{the number of edges of }P_{1}\left(x\right)\right)^{2}\right]<\infty.
\]
Now assume that the first statement of Proposition \ref{prop:result on bddness of second moment of number of edges}
holds, i.e., the second moment of the number of edges of $A_{n}\left(x\right)$
is bounded in $n$. By taking $G\left(\cdot\right)$ to be the number
of edges of its input in Theorem \ref{thm:convergence in distribution main result in Section 2},
one can easily check that 
\[
\lim\limits _{n\rightarrow\infty}\mathbb{E}\left[\text{the number of edges of }A_{n}\left(x\right)\right]=\lim\limits _{n\rightarrow\infty}\mathbb{E}\left[\text{the number of edges of }P_{n}\left(x\right)\right]=6.
\]
Therefore, we only need to focus on proving the first statement of
Proposition \ref{prop:result on bddness of second moment of number of edges}. 

To this end, we first observe that since every vertex of $A_{n}\left(x\right)$
has degree 3 in the Voronoi diagram, every vertex of $A_{n}\left(x\right)$
is also shared by a unique pair of adjacent Voronoi cells of $A_{n}\left(x\right)$.
For all $i,j\in\left\{ 1,\cdots,n\right\} $ such that $i\neq j$,
let $E_{i,j}$ denote the event that $A_{n}\left(X_{i}\right)$, $A_{n}\left(X_{j}\right)$
and $A_{n}\left(x\right)$ share a vertex. Then, we have that 
\[
\begin{split}\mathbb{E}\left[\left(\text{the number of edges of }A_{n}\left(x\right)\right)^{2}\right] & =\mathbb{E}\left[\left(\text{the number of vertices of }A_{n}\left(x\right)\right)^{2}\right]\\
 & =\mathbb{E}\left[\left(\sum_{i,j\in\{1,\cdots,n\},\ i<j}\mathbb{I}_{E_{i,j}}\right)^{2}\right].
\end{split}
\]
Let $i$, $j$, $k$ and $l$ be four distinct integers in $\left\{ 1,\cdots,n\right\} $.
There are three types of terms resulting from the expectation above:\\

\noindent 1. There are $\frac{n\left(n-1\right)}{2}=\mathcal{O}\left(n^{2}\right)$
terms of the form 
\[
\mathbb{E}\left[\left(\mathbb{I}_{E_{i,j}}\right)^{2}\right]=\mathbb{P}\left(E_{i,j}\right).
\]
2. There are $\frac{n\left(n-1\right)}{2}\cdot2\left(n-2\right)=\mathcal{O}\left(n^{3}\right)$
terms of the form 
\[
\mathbb{E}\left[\mathbb{I}_{E_{i,j}}\mathbb{I}_{E_{i,k}}\right]=\mathbb{P}\left(E_{i,j}\cap E_{i,k}\right).
\]
3. There are $\frac{n\left(n-1\right)}{2}\cdot\frac{\left(n-2\right)\left(n-3\right)}{2}=\mathcal{O}\left(n^{4}\right)$
terms of the form 
\[
\mathbb{E}\left[\mathbb{I}_{E_{i,j}}\mathbb{I}_{E_{l,k}}\right]=\mathbb{P}\left(E_{i,j}\cap E_{l,k}\right).
\]
In particular, since $X_{1},\cdots,X_{n}$ are all identically distributed,
it is enough to show that the following three quantities are finite:
\[
\sup_{n\geq2}n^{2}\mathbb{P}\left(E_{1,2}\right),\:\sup_{n\geq3}n^{3}\mathbb{P}\left(E_{1,2}\cap E_{1,3}\right)\text{ and }\sup_{n\geq4}n^{4}\mathbb{P}\left(E_{1,2}\cap E_{3,4}\right).
\]
We will only give an explicit proof that the last quantity is finite.
The finiteness of the first two quantities can be shown in an identical
manner.

For any three points $w,u,v\in\mathbb{R}^{2}$ that do not lie on
the same line, let $c\left(w,u,v\right)$ denote the circumcenter
of $w$, $u$ and $v$. Observe that $E_{i,j}$ is equivalent to the
event that 
\[
\left\Vert X_{k}-c\left(x,X_{i},X_{j}\right)\right\Vert \geq\left\Vert x-c\left(x,X_{i},X_{j}\right)\right\Vert \,\forall k\in\left\{ 1,\cdots,n\right\} \backslash\left\{ i,j\right\} .
\]
To see this, notice that the vertex shared by $A_{n}\left(x\right)$,
$A_{n}\left(X_{i}\right)$, and $A_{n}\left(X_{j}\right)$ is just
$c\left(x,X_{i},X_{j}\right)$, and $c\left(x,X_{i},X_{j}\right)$
is in all three of $A_{n}\left(x\right)$, $A_{n}\left(X_{i}\right)$,
and $A_{n}\left(X_{j}\right)$ if and only if $c\left(x,X_{i},X_{j}\right)$
is not strictly closer to any of the other points in $X_{1},\cdots,X_{n}$.
Therefore, we have that
\[
\begin{split} & \mathbb{P}\left(E_{1,2}\cap E_{3,4}\right)\\
\le & \mathbb{P}\left(\bigcap_{k=5}^{n}\left\{ \left\Vert X_{k}-c\left(x,X_{1},X_{2}\right)\right\Vert \geq\left\Vert x-c\left(x,X_{1},X_{2}\right)\right\Vert ,\left\Vert X_{k}-c\left(x,X_{3},X_{4}\right)\right\Vert \geq\left\Vert x-c\left(x,X_{3},X_{4}\right)\right\Vert \right\} \right)\\
= & \mathbb{E}\left[\left(1-\mu_{f}\left(B_{c\left(x,X_{1},X_{2}\right),\left\Vert x-c\left(x,X_{1},X_{2}\right)\right\Vert }\cup B_{c\left(x,X_{3},X_{4}\right),\left\Vert x-c\left(x,X_{3},X_{4}\right)\right\Vert }\right)\right)^{n-4}\right].
\end{split}
\]
We claim that in order to prove that 
\[
\sup_{n\geq4}n^{4}\mathbb{E}\left[\left(1-\mu_{f}\left(B_{c\left(x,X_{1},X_{2}\right),\left\Vert x-c\left(x,X_{1},X_{2}\right)\right\Vert }\cup B_{c\left(x,X_{3},X_{4}\right),\left\Vert x-c\left(x,X_{3},X_{4}\right)\right\Vert }\right)\right)^{n-4}\right]<\infty,
\]
it is enough to show that 
\begin{equation}
\limsup_{z\rightarrow0}z^{-4}\mathbb{P}\left(\mu_{f}\left(B_{c\left(x,X_{1},X_{2}\right),\left\Vert x-c\left(x,X_{1},X_{2}\right)\right\Vert }\cup B_{c\left(x,X_{3},X_{4}\right),\left\Vert x-c\left(x,X_{3},X_{4}\right)\right\Vert }\right)\leq z\right)<\infty.\label{eq:decay estimate of prob involving E1,2 and E3,4}
\end{equation}
Suppose (\ref{eq:decay estimate of prob involving E1,2 and E3,4})
holds. Then, there exists $M>0$ and $\eta>0$, such that $\forall z\in\left(0,\eta\right)$,
\[
\mathbb{P}\left(\mu_{f}\left(B_{c\left(x,X_{1},X_{2}\right),\left\Vert x-c\left(x,X_{1},X_{2}\right)\right\Vert }\cup B_{c\left(x,X_{3},X_{4}\right),\left\Vert x-c\left(x,X_{3},X_{4}\right)\right\Vert }\right)\leq z\right)\leq z^{4}M.
\]
Therefore, we have that 
\[
\begin{split} & n^{4}\mathbb{E}\left[\left(1-\mu_{f}\left(B_{c\left(x,X_{1},X_{2}\right),\left\Vert x-c\left(x,X_{1},X_{2}\right)\right\Vert }\cup B_{c\left(x,X_{3},X_{4}\right),\left\Vert x-c\left(x,X_{3},X_{4}\right)\right\Vert }\right)\right)^{n-4}\right]\\
= & n^{4}\left(n-4\right)\int_{0}^{1}\left(1-z\right)^{n-5}\mathbb{P}\left(\mu_{f}\left(B_{c\left(x,X_{1},X_{2}\right),\left\Vert x-c\left(x,X_{1},X_{2}\right)\right\Vert }\cup B_{c\left(x,X_{3},X_{4}\right),\left\Vert x-c\left(x,X_{3},X_{4}\right)\right\Vert }\right)\leq z\right)dz\\
\leq & n^{4}\left(n-4\right)\int_{0}^{\eta}\left(1-z\right)^{n-5}\mathbb{P}\left(\mu_{f}\left(B_{c\left(x,X_{1},X_{2}\right),\left\Vert x-c\left(x,X_{1},X_{2}\right)\right\Vert }\cup B_{c\left(x,X_{3},X_{4}\right),\left\Vert x-c\left(x,X_{3},X_{4}\right)\right\Vert }\right)\leq z\right)dz\\
 & \hspace{10cm}+n^{4}\left(n-4\right)\int_{\eta}^{1}\left(1-z\right)^{n-5}dz\\
\leq & Mn^{4}\left(n-4\right)\int_{0}^{\eta}\left(1-z\right)^{n-5}z^{4}dz+n^{4}\left(1-\eta\right)^{n-4}\\
= & \frac{4!Mn^{3}}{\left(n-1\right)\left(n-2\right)\left(n-3\right)}+\mathcal{O}\left(n^{4}\left(1-\eta\right)^{n-4}\right),
\end{split}
\]
which is clearly bounded in $n$. Now we will focus on proving (\ref{eq:decay estimate of prob involving E1,2 and E3,4}).

By the generalized Lebesgue Differentiation Theorem (see, e.g., page
42 of \cite{Devroye2001}), there exists $\delta>0$ such that $\forall y\in B_{x,\delta}$,
\begin{equation}
\left|\frac{\mu_{f}\left(B_{x,\left\Vert y-x\right\Vert }\right)}{\lambda\left(B_{x,\left\Vert y-x\right\Vert }\right)}-f\left(x\right)\right|\leq\frac{1}{2}f\left(x\right)\text{ and }\left|\frac{\mu_{f}\left(B_{y,\left\Vert y-x\right\Vert }\right)}{\lambda\left(B_{y,\left\Vert y-x\right\Vert }\right)}-f\left(x\right)\right|\leq\frac{1}{2}f\left(x\right).\label{eq:Leb diff theorem in the controlling off centered balls}
\end{equation}
We claim that if $\left\Vert x-c\left(x,X_{1},X_{2}\right)\right\Vert >\frac{\delta}{2}$,
then it must be that 
\[
\mu_{f}\left(B_{c\left(x,X_{1},X_{2}\right),\left\Vert x-c\left(x,X_{1},X_{2}\right)\right\Vert }\right)\geq\frac{1}{8}f\left(x\right)\pi\delta^{2}.
\]
To see this, choose the point $c^{*}$ on the line segment connecting
$x$ and $c\left(x,X_{1},X_{2}\right)$ such that $\left\Vert x-c^{*}\right\Vert =\frac{\delta}{2}$.
Then, $B_{c^{*},\frac{\delta}{2}}\subseteq B_{c\left(x,X_{1},X_{2}\right),\left\Vert x-c\left(x,X_{1},X_{2}\right)\right\Vert }$
as shown in Figure 2.1.

\begin{figure}[H]
\includegraphics[scale=0.5]{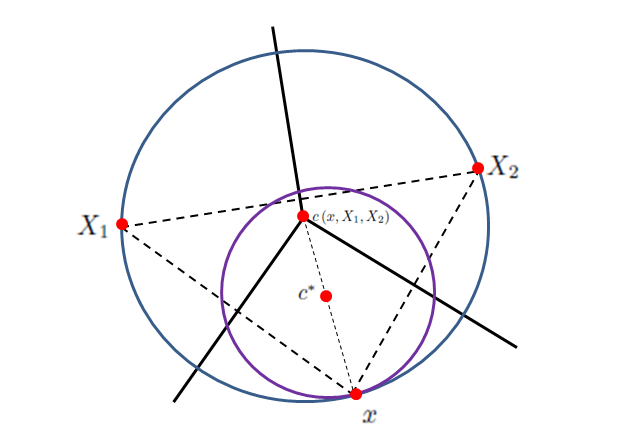}\caption{$A_{n}\left(x\right)$, $A_{n}\left(X_{1}\right)$ and $A_{n}\left(X_{2}\right)$
sharing a vertex $c\left(x,X_{1},X_{2}\right)$.}

\end{figure}
 By (\ref{eq:Leb diff theorem in the controlling off centered balls}),
we have that 
\[
\mu_{f}\left(B_{c\left(x,X_{1},X_{2}\right),\left\Vert x-c\left(x,X_{1},X_{2}\right)\right\Vert }\right)\geq\mu_{f}\left(B_{c^{*},\frac{\delta}{2}}\right)\geq\frac{1}{2}f\left(x\right)\lambda(B_{c^{*},\frac{\delta}{2}})=\frac{1}{8}f\left(x\right)\pi\delta^{2}.
\]
Meanwhile, it is clear that
\[
\frac{1}{2}\left\Vert x-X_{1}\right\Vert \leq\left\Vert x-c\left(x,X_{1},X_{2}\right)\right\Vert \text{ and }\frac{1}{2}\left\Vert x-X_{2}\right\Vert \leq\left\Vert x-c\left(x,X_{1},X_{2}\right)\right\Vert .
\]
Therefore, if 
\[
\mu_{f}\left(B_{c\left(x,X_{1},X_{2}\right),\left\Vert x-c\left(x,X_{1},X_{2}\right)\right\Vert }\right)<\frac{1}{8}f\left(x\right)\pi\delta^{2},
\]
then it means that
\[
\left\Vert x-c\left(x,X_{1},X_{2}\right)\right\Vert \leq\frac{\delta}{2}
\]
and hence 
\[
\left\Vert x-X_{1}\right\Vert \leq\delta\text{ and }\left\Vert x-X_{2}\right\Vert \leq\delta.
\]
Similarly, we also conclude that if
\[
\mu_{f}\left(B_{c\left(x,X_{3},X_{4}\right),\left\Vert x-c\left(x,X_{3},X_{4}\right)\right\Vert }\right)<\frac{1}{8}f\left(x\right)\pi\delta^{2},
\]
then
\[
\left\Vert x-c\left(x,X_{3},X_{4}\right)\right\Vert \leq\frac{\delta}{2},\;\left\Vert x-X_{3}\right\Vert \leq\delta\text{ and }\left\Vert x-X_{4}\right\Vert \leq\delta.
\]
Let $E$ be the set that
\[
\begin{array}{c}
\left\{ \left\Vert x-c\left(x,X_{1},X_{2}\right)\right\Vert <\frac{\delta}{2},\left\Vert x-X_{1}\right\Vert \leq\delta,\left\Vert x-X_{2}\right\Vert \leq\delta\right\} \hspace{2cm}\\
\hspace{2cm}\bigcap\left\{ \left\Vert x-c\left(x,X_{3},X_{4}\right)\right\Vert \leq\frac{\delta}{2},\left\Vert x-X_{3}\right\Vert \leq\delta,\left\Vert x-X_{4}\right\Vert \leq\delta\right\} .
\end{array}
\]

We observe that, by (\ref{eq:Leb diff theorem in the controlling off centered balls}),
for every $y\in B_{x,\delta}$, 
\[
\mu_{f}\left(B_{x,\left\Vert y-x\right\Vert }\right)\leq\frac{3}{2}f\left(x\right)\lambda\left(B_{x,\left\Vert y-x\right\Vert }\right)=\frac{3}{2}f\left(x\right)\lambda\left(B_{y,\left\Vert y-x\right\Vert }\right)\leq3\mu_{f}\left(B_{y,\left\Vert y-x\right\Vert }\right)
\]
and similarly,
\[
\mu_{f}\left(B_{x,\left\Vert y-x\right\Vert }\right)\leq\frac{3}{2}f\left(x\right)\lambda\left(B_{x,\left\Vert y-x\right\Vert }\right)=3\cdot2^{d-1}f\left(x\right)\lambda\left(B_{x,\frac{1}{2}\left\Vert y-x\right\Vert }\right)\leq3\cdot2^{d}\mu_{f}\left(B_{x,\frac{1}{2}\left\Vert y-x\right\Vert }\right).
\]
Then, we have that for all $z<\frac{1}{8}f\left(x\right)\pi\delta^{2}$,
\[
\begin{split} & z^{-4}\mathbb{P}\left(\mu_{f}\left(B_{c\left(x,X_{1},X_{2}\right),\left\Vert x-c\left(x,X_{1},X_{2}\right)\right\Vert }\cup B_{c\left(x,X_{3},X_{4}\right),\left\Vert x-c\left(x,X_{3},X_{4}\right)\right\Vert }\right)\leq z\right)\\
\leq & z^{-4}\mathbb{P}\left(\max\left\{ \mu_{f}\left(B_{c\left(x,X_{1},X_{2}\right),\left\Vert x-c\left(x,X_{1},X_{2}\right)\right\Vert }\right),\mu_{f}\left(B_{c\left(x,X_{3},X_{4}\right),\left\Vert x-c\left(x,X_{3},X_{4}\right)\right\Vert }\right)\right\} \leq z\right)\\
= & z^{-4}\mathbb{P}\left(\max\left\{ \mu_{f}\left(B_{c\left(x,X_{1},X_{2}\right),\left\Vert x-c\left(x,X_{1},X_{2}\right)\right\Vert }\right),\mu_{f}\left(B_{c\left(x,X_{3},X_{4}\right),\left\Vert x-c\left(x,X_{3},X_{4}\right)\right\Vert }\right)\right\} \leq z,\,E\right)\\
\leq & z^{-4}\mathbb{P}\left(\max\left\{ \mu_{f}\left(B_{x,\left\Vert x-c\left(x,X_{1},X_{2}\right)\right\Vert }\right),\mu_{f}\left(B_{x,\left\Vert x-c\left(x,X_{3},X_{4}\right)\right\Vert }\right)\right\} \leq3z,\,E\right)\\
\leq & z^{-4}\mathbb{P}\left(\max\left\{ \mu_{f}\left(B_{x,\frac{1}{2}\left\Vert x-X_{1}\right\Vert }\right),\mu_{f}\left(B_{x,\frac{1}{2}\left\Vert x-X_{2}\right\Vert }\right),\mu_{f}\left(B_{x,\frac{1}{2}\left\Vert x-X_{3}\right\Vert }\right),\mu_{f}\left(B_{x,\frac{1}{2}\left\Vert x-X_{4}\right\Vert }\right)\right\} \leq3z,\,E\right)\\
\leq & z^{-4}\mathbb{P}\left(\max\left\{ \mu_{f}\left(B_{x,\left\Vert x-X_{1}\right\Vert }\right),\mu_{f}\left(B_{x,\left\Vert x-X_{2}\right\Vert }\right),\mu_{f}\left(B_{x,\left\Vert x-X_{3}\right\Vert }\right),\mu_{f}\left(B_{x,\left\Vert x-X_{4}\right\Vert }\right)\right\} \leq9\cdot2^{d}z,\,E\right)\\
\leq & z^{-4}\mathbb{P}\left(\max\left\{ \mu_{f}\left(B_{x,\left\Vert x-X_{1}\right\Vert }\right),\mu_{f}\left(B_{x,\left\Vert x-X_{2}\right\Vert }\right),\mu_{f}\left(B_{x,\left\Vert x-X_{3}\right\Vert }\right),\mu_{f}\left(B_{x,\left\Vert x-X_{4}\right\Vert }\right)\right\} \leq9\cdot2^{d}z\right)\\
\leq & z^{-4}\mathbb{P}\left(\max\left\{ I_{1},I_{2},I_{3},I_{4}\right\} \leq9\cdot2^{d}z\right)\\
= & \left(9\cdot2^{d}\right)^{4},
\end{split}
\]
where $I_{1},I_{2},I_{3}$ and $I_{4}$ are i.i.d. random variables
with the uniform distribution on $\left[0,1\right]$. We conclude
that 
\[
\sup_{z\in\left(0,\frac{1}{8}f\left(x\right)\pi\delta^{2}\right)}z^{-4}\mathbb{P}\left(\mu_{f}\left(B_{c\left(x,X_{1},X_{2}\right),\left\Vert x-c\left(x,X_{1},X_{2}\right)\right\Vert }\cup B_{c\left(x,X_{3},X_{4}\right),\left\Vert x-c\left(x,X_{3},X_{4}\right)\right\Vert }\right)\leq z\right)\leq\left(9\cdot2^{d}\right)^{4}
\]
which gives the desired result (\ref{eq:decay estimate of prob involving E1,2 and E3,4}).
\end{proof}

\section{Voronoi Cells That Contain a Fixed Point}

We now shift our focus to the consideration of the Voronoi cell, denoted
by $L_{n}\left(x\right)$, that contains the fixed point $x\in\mathbb{R}^{d}$
in the Voronoi diagram generated by $\left\{ X_{1},\cdots,X_{n}\right\} $
for every $n\geq1$. In other words, $x$ will almost surely never
be the nucleus of $L_{n}\left(x\right)$ for any $n\geq1$, and $L_{n}\left(x\right)$'s
nucleus may vary as $n$ varies. In general, we expect $L_{n}\left(x\right)$
to behave similarly to $A_{n}\left(x\right)$, but to ``tend'' to
have larger measure under $\mu_{f}$. Heuristically speaking, by requiring
the cell to contain a fixed point, we are in some sense biasing our
selection towards larger cells. We begin our study by giving an estimate
on $D_{n}^{L}\left(x\right)$, which we recall is the diameter of
$L_{n}\left(x\right)$, that mirrors the result for $A_{n}\left(x\right)$
obtained in Theorem 5.1 of \cite{Devroye2017}.\\

\noindent \textbf{Theorem \ref{thm:diameter control on L_n(x), first result in section 3}.}
\emph{Let $f$ be a probability density function on $\mathbb{R}^{d}$
and $x$ be a Lebesgue point of $f$ such that $f\left(x\right)>0$.
Then, there exist universal constants $c_{1},c_{2}>0$ such that $\forall t>0$,
\[
\limsup\limits _{n\rightarrow\infty}\mathbb{P}\left(D_{n}^{L}\left(x\right)\geq\frac{t}{n^{\frac{1}{d}}}\right)\leq c_{1}e^{-c_{2}f\left(x\right)t^{d}}.
\]
}

Before giving the proof of Theorem \ref{thm:diameter control on L_n(x), first result in section 3},
we will first state a technical lemma whose proof is left in the Appendix.\\

\noindent \textbf{Lemma \ref{lem:Appendix cone arguments on range of cell}
}(in the Appendix). Let $\alpha>0$, $x\in\mathbb{R}^{d}$, and $C\subseteq\mathbb{R}^{d}$
be any cone of angle $\frac{\pi}{12}$ and with origin at $x$ ($C$
does not contain $x$), i.e., 
\[
C:=\left\{ y\in\mathbb{R}^{d}\backslash\left\{ x\right\} :\frac{\left(v,y-x\right)_{\mathbb{R}^{d}}}{\left\Vert y-x\right\Vert }\geq\cos\left(\frac{\pi}{24}\right)\right\} \text{, for some }v\in\mathbb{R}^{d}\text{ with }\left\Vert v\right\Vert =1.
\]
Let $R_{1}=\frac{1}{64}\alpha$, $R_{2}=\frac{1+31\cos\left(\frac{\pi}{6}\right)}{64\cos\left(\frac{\pi}{12}\right)}\alpha$,
and $R_{3}=\frac{30}{64}\alpha$. Then, for any $p,y,z\in C$, if
\[
0<\left\Vert y-x\right\Vert <R_{1},\ R_{2}\leq\left\Vert p-x\right\Vert <R_{3},\text{ and }\left\Vert x-z\right\Vert \geq\frac{\alpha}{2},
\]
then we must have 
\[
\left\Vert z-p\right\Vert <\left\Vert z-y\right\Vert .
\]
\begin{proof}
\noindent \textit{(Proof of Theorem \ref{thm:diameter control on L_n(x), first result in section 3}.)}
Let $t>0$ and $C_{1},\cdots,C_{\gamma_{d}}$ be a minimal set of
cones of angle $\frac{\pi}{12}$ and with origin at $x$ such that
their union covers $\mathbb{R}^{d}$. For every $i\in\left\{ 1,\cdots,\gamma_{d}\right\} $
and $n\geq1$ define the following three sections of $C_{i}$: 
\[
C_{i}^{1,n}:=\left\{ z\in C_{i}:\left\Vert z-x\right\Vert <R_{1,n}\right\} ,
\]
\[
C_{i}^{2,n}:=\left\{ z\in C_{i}:R_{1,n}\leq\left\Vert z-x\right\Vert <R_{2,n}\right\} 
\]
and
\[
C_{i}^{3,n}:=\left\{ z\in C_{i}:R_{2,n}\leq\left\Vert z-x\right\Vert <R_{3,n}\right\} ,
\]
where $R_{1,n}$, $R_{2,n}$ and $R_{3,n}$ are respectively $R_{1}$,
$R_{2}$ and $R_{3}$ defined as in Lemma \ref{lem:Appendix cone arguments on range of cell}
with $\alpha:=\frac{t}{n^{\frac{1}{d}}}$.

Now, for every $n\geq1$, suppose that for every $i\in\left\{ 1,\cdots,\gamma_{d}\right\} $,
$\exists p_{i}\in\left\{ X_{1},\cdots,X_{n}\right\} \cap C_{i}^{3,n}$.
Let $y$ denote the nucleus of $L_{n}\left(x\right)$ and assume that
$\left\Vert y-x\right\Vert <R_{1,n}$. We claim that $L_{n}\left(x\right)\subseteq B_{x,\frac{t}{2n^{\frac{1}{d}}}}$
and hence $D_{n}^{L}\left(x\right)<\frac{t}{n^{\frac{1}{d}}}$. To
see this, let $z\in\mathbb{R}^{d}$ be such that $\left\Vert z-x\right\Vert \geq\frac{t}{2n^{\frac{1}{d}}}$.
Let $y^{\prime}$ be the point on the line segment from $x$ to $z$
such that $\left\Vert y-x\right\Vert =\left\Vert y^{\prime}-x\right\Vert $
and $i_{0}\in\left\{ 1,\cdots,\gamma_{d}\right\} $ be such that $z\in C_{i_{0}}$.
We clearly have $\left\Vert y-z\right\Vert \geq\left\Vert y^{\prime}-z\right\Vert $.
Moreover, Lemma \ref{lem:Appendix cone arguments on range of cell}
immediately gives that $\left\Vert z-p_{i_{0}}\right\Vert <\left\Vert z-y^{\prime}\right\Vert $.
Hence, $\left\Vert z-p_{i_{0}}\right\Vert <\left\Vert z-y\right\Vert \implies z\notin L_{n}\left(x\right)$,
as desired. We conclude that $\forall n\geq1$, 
\[
\mathbb{P}\left(D_{n}^{L}\left(x\right)\geq\frac{t}{n^{\frac{1}{d}}}\right)\leq\mathbb{P}\left(\left\Vert y-x\right\Vert \geq R_{1,n}\text{ or }\exists\ i\in\left\{ 1,\cdots,\gamma_{d}\right\} \text{ such that }C_{i}^{3,n}\cap\left\{ X_{1},\cdots,X_{n}\right\} =\emptyset\right).
\]
Hence it is enough to control
\[
\limsup_{n\rightarrow\infty}\mathbb{P}\left(\left\Vert y-x\right\Vert \geq R_{1,n}\text{ or }\exists\ i\in\left\{ 1,\cdots,\gamma_{d}\right\} \text{ such that }C_{i}^{3,n}\cap\left\{ X_{1},\cdots,X_{n}\right\} =\emptyset\right).
\]

By the generalized Lebesgue Differentiation Theorem, we have that
for every $i\in\left\{ 1,\cdots,\gamma_{d}\right\} $ and all $n$
sufficiently large, 
\[
\left|\frac{\mu_{f}\left(C_{i}^{3,n}\right)}{\lambda\left(C_{i}^{3,n}\right)}-f\left(x\right)\right|\leq\frac{1}{2}f\left(x\right)\text{ and }\left|\frac{\mu_{f}\left(B_{x,R_{1,n}}\right)}{\lambda\left(B_{x,R_{1,n}}\right)}-f\left(x\right)\right|\leq\frac{1}{2}f\left(x\right).
\]
It follows that
\[
\begin{split} & \mathbb{P}\left(\left\Vert y-x\right\Vert \geq R_{1,n}\text{ or }\exists\ i\in\left\{ 1,\cdots,\gamma_{d}\right\} \text{ such that }C_{i}^{3,n}\cap\left\{ X_{1},\cdots,X_{n}\right\} =\emptyset\right)\\
\leq & \sum_{i=1}^{\gamma_{d}}\mathbb{P}\left(C_{i}^{3,n}\cap\left\{ X_{1},\cdots,X_{n}\right\} =\emptyset\right)+\mathbb{P}\left(\left\Vert y-x\right\Vert \geq R_{1,n}\right)\\
= & \sum_{i=1}^{\gamma_{d}}\left(1-\mu_{f}\left(C_{i}^{3,n}\right)\right)^{n}+\left(1-\mu_{f}\left(B_{x,R_{1,n}}\right)\right)^{n}\\
\leq & \gamma_{d}\left(1-\frac{1}{2}\lambda\left(C_{i}^{3,n}\right)f\left(x\right)\right)^{n}+\left(1-\frac{1}{2}\lambda\left(B_{x,R_{1,n}}\right)f\left(x\right)\right)^{n}\\
\leq & c_{1}\left(1-\frac{c_{2}t^{d}f\left(x\right)}{n}\right)^{n}\rightarrow c_{1}e^{-c_{2}t^{d}f\left(x\right)}\text{ as }n\rightarrow\infty,
\end{split}
\]
where $c_{1}$ and $c_{2}$ are two universal positive constants. 
\end{proof}
\noindent We now examine the relationship between $f$ and $\mu_{f}\left(L_{n}\left(x\right)\right)$
and establish Theorem \ref{thm:limiting distribution of n*L_n(x), second result in section 3},
which states that under the same conditions as the ones we have been
imposing on $x$ (i.e., $x$ is a Lebesgue point of $f$ and $f\left(x\right)>0$),
$n\mu_{f}\left(L_{n}\left(x\right)\right)$ converges in distribution
to a random variable whose distribution is universal for all choices
of $f$. We provide a complete characterization of this limiting distribution
in terms of its moments and show that these moments determines a unique
distribution. Furthermore, since the limiting distribution is independent
of $f$, we are able to obtain information about this limiting distribution
by numerically simulating data from the special case where $f$ is
the probability density function of the uniform distribution on $\left[-1,1\right]^{2}$.
A histogram estimate of the density of the limiting distribution is
shown (in copper) in Figure 3.1. For the purpose of comparison, we
also simulate data to estimate the probability density function of
the limiting distribution of $n\mu_{f}\left(A_{n}\left(x\right)\right)$
(i.e., the case when $x$ is the nucleus of the cell) derived in \cite{Devroye2017}
and also place the histogram (in blue) in Figure 3.1. We see that
as expected, the limiting distribution of $n\mu_{f}\left(L_{n}\left(x\right)\right)$
gives higher probabilities to larger values than the comparative distribution
for $n\mu_{f}\left(A_{n}\left(x\right)\right)$.
\noindent \begin{center}
\begin{figure}[H]
\includegraphics[scale=0.22]{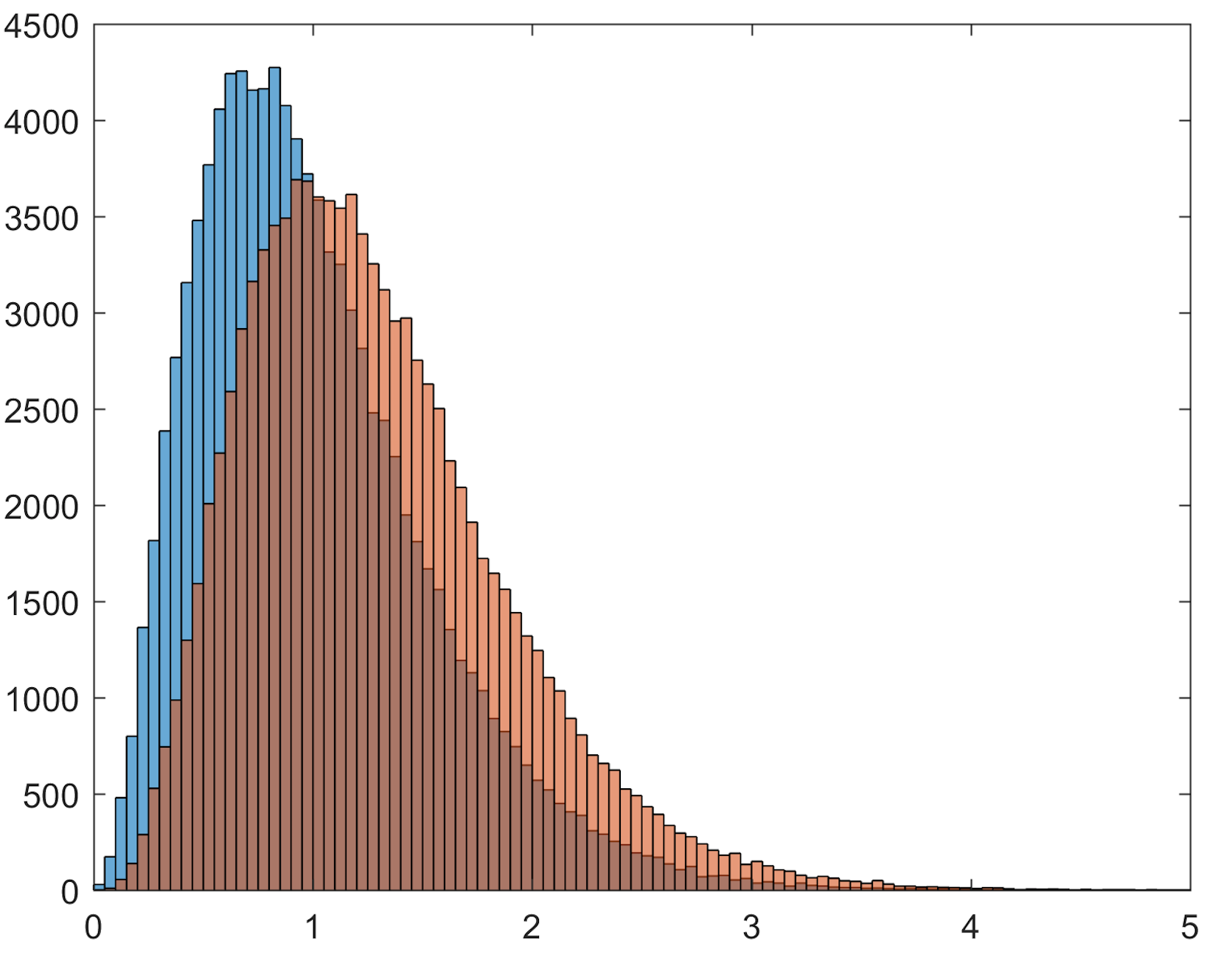}\caption{\textbf{\footnotesize{}Histogram of the density estimate of the limiting
distribution of $n\mu_{f}\left(L_{n}\left(x\right)\right)$ (copper)
and $n\mu_{f}\left(A_{n}\left(x\right)\right)$ (blue).}}
\noindent \raggedright{}{\footnotesize{}One thousand samples of the
Voronoi diagram arising from the point process $\left\{ X_{1},\cdots,X_{1000}\right\} $
that is composed of i.i.d. random variables with the uniform distribution
on $\left[-1,1\right]^{2}$ were taken. $n\mu_{f}\left(L_{n}\left(0\right)\right)$
and $n\mu_{f}\left(A_{n}\left(0\right)\right)$ were calculated for
each trial and the resulting data was grouped together into bins of
width 0.05. The $x$-axis indicates the observed values for $n\mu_{f}\left(L_{n}\left(0\right)\right)$
and $n\mu_{f}\left(A_{n}\left(0\right)\right)$ while the $y$-axis
shows the number of occurrences of values in each bin.} {\footnotesize{}The
authors are grateful to Jean-Christophe Nave for performing the simulations
of the two limiting distributions and generating the histogram figures.}{\footnotesize \par}
\end{figure}
\par\end{center}

Below let us state Theorem \ref{thm:limiting distribution of n*L_n(x), second result in section 3}
once again.\\

\noindent \textbf{Theorem \ref{thm:limiting distribution of n*L_n(x), second result in section 3}.}
\emph{For every positive integer $k\geq1$, let $W$ be a Bernoulli$\left(\frac{k}{k+1}\right)$
random variable and $U_{1},\cdots,U_{k}$ be i.i.d. random variables
with the uniform distribution on $B_{0,1}$ that are independent of
$W$. Set $\bar{1}:=\left(1,0,\cdots,0\right)\in\mathbb{R}^{d}$,
and define the random variable }
\[
\begin{split}D_{k}:= & \frac{\lambda\left(B_{U_{1},||\bar{1}-U_{1}||}\cup\dots\cup B_{U_{k},||\bar{1}-U_{k}||}\cup B_{0,1}\right)}{\lambda\left(B_{0,1}\right)}\mathbb{I}_{\left\{ W=0\right\} }\\
 & +\frac{\lambda\left(B_{\bar{1},||\bar{1}-U_{1}||}\cup B_{U_{2},||U_{1}-U_{2}||}\cup B_{U_{3},||U_{1}-U_{3}||}\cup\dots\cup B_{U_{k},||U_{1}-U_{k}||}\cup B_{0,||U_{1}||}\right)}{\lambda\left(B_{0,1}\right)}\mathbb{I}_{\left\{ W=1\right\} }.
\end{split}
\]
\emph{Let $f$ be any probability density function on $\mathbb{R}^{d}$
and $x$ be a Lebesgue point of $f$ such that $f\left(x\right)>0$.
Then, 
\[
\lim\limits _{n\rightarrow\infty}\mathbb{E}\left[n^{k}\mu_{f}\left(L_{n}\left(x\right)\right)^{k}\right]=\mathbb{E}\left[\frac{(k+1)!}{D_{k}^{k+1}}\right],\ \forall k\geq1.
\]
Moreover, these limits of the moments uniquely determine a distribution
$\mathscr{D}$ on $\mathbb{R}^{+}$ with the property that $\mathscr{D}$
does not depend on the choice of $f$ or $x$, and the distribution
of $n\mu_{f}\left(L_{n}\left(x\right)\right)$ weakly converges to
$\mathscr{D}$ as $n\rightarrow\infty$. }
\begin{proof}
We only give a detailed proof for the case $k=1$. Higher moments
can be dealt with using similar methods without any additional steps.
We split the proof into five main parts.\\

\noindent \textbf{Step 1. Reduce estimating $\mathbb{E}\left[n\mu_{f}\left(L_{n}\left(x\right)\right)\right]$
to estimating a tail probability:} For every $n\geq1$, let $X_{n+1}$
be a random variable with probability density function $f$, $X_{n+1}$
be independent of $X_{1},\cdots,X_{n}$, and $y$ be the nucleus of
$L_{n}\left(x\right)$. We have that 
\[
\begin{split}\mathbb{E}\left[\mu_{f}\left(L_{n}\left(x\right)\right)\right] & =\mathbb{P}\left(X_{n+1}\in L_{n}\left(x\right)\right)\\
 & =\mathbb{P}\left(\left\Vert X_{n+1}-y\right\Vert \leq\left\Vert X_{n+1}-X_{i}\right\Vert \text{ for every }i\in\left\{ 1,\cdots,n\right\} \text{ such that }X_{i}\neq y\right)\\
 & =\mathbb{P}\left(X_{i}\notin B_{X_{n+1},\left\Vert X_{n+1}-y\right\Vert }\text{ for every }i\in\left\{ 1,\cdots,n\right\} \text{ such that }X_{i}\neq y\right)\\
 & =n\mathbb{P}\left(X_{1}=y,\:X_{i}\notin B_{X_{n+1},\left\Vert X_{n+1}-X_{1}\right\Vert }\text{ for every }i\in\left\{ 2,\cdots,n\right\} \right)\\
 & =n\mathbb{P}\left(X_{i}\notin B_{X_{n+1},\left\Vert X_{n+1}-X_{1}\right\Vert }\cup B_{x,\left\Vert x-X_{1}\right\Vert }\text{ for every }i\in\left\{ 2,\cdots,n\right\} \right)\\
 & =n\mathbb{E}\left[\left(1-\mu_{f}\left(B_{X_{n+1},\left\Vert X_{n+1}-X_{1}\right\Vert }\cup B_{x,\left\Vert x-X_{1}\right\Vert }\right)\right)^{n-1}\right].
\end{split}
\]
Now, using arguments that are identical to those employed in the derivation
of (\ref{eq:decay estimate of prob involving E1,2 and E3,4}) in the
proof of Proposition \ref{prop:result on bddness of second moment of number of edges},
we know that in order to prove that 
\begin{equation}
\lim_{n\rightarrow\infty}n^{2}\mathbb{E}\left[\left(1-\mu_{f}\left(B_{X_{n+1},\left\Vert X_{n+1}-X_{1}\right\Vert }\cup B_{x,\left\Vert x-X_{1}\right\Vert }\right)\right)^{n-1}\right]=\mathbb{E}\left[\frac{2}{D_{1}^{2}}\right],\label{eq:theorem result for 1st moment}
\end{equation}
it is sufficient to prove that 
\begin{equation}
\lim_{z\rightarrow0}z^{-2}\mathbb{P}\left(\mu_{f}\left(B_{X_{n+1},\left\Vert X_{n+1}-X_{1}\right\Vert }\cup B_{x,\left\Vert x-X_{1}\right\Vert }\right)\leq z\right)=\mathbb{E}\left[\frac{1}{D_{1}^{2}}\right].\label{eq:prob to be estimated step 1}
\end{equation}
\textbf{Step 2. Simplify the probability to be estimated:} For every
$z>0$, we rewrite the concerned probability in (\ref{eq:prob to be estimated step 1})
as 
\begin{equation}
\mathbb{P}\left(\frac{\mu_{f}\left(B_{X_{n+1},\left\Vert X_{n+1}-X_{1}\right\Vert }\cup B_{x,\left\Vert x-X_{1}\right\Vert }\right)}{\max\left\{ \mu_{f}\left(B_{x,\left\Vert X_{n+1}-x\right\Vert }\right),\mu_{f}\left(B_{x,\left\Vert x-X_{1}\right\Vert }\right)\right\} }\cdot\max\left\{ \mu_{f}\left(B_{x,\left\Vert X_{n+1}-x\right\Vert }\right),\mu_{f}\left(B_{x,\left\Vert x-X_{1}\right\Vert }\right)\right\} \leq z\right).\label{eq:rewriting prob to be estimated step 2}
\end{equation}
Now, let $I_{1}$ and $I_{2}$ be i.i.d. random variables with the
uniform distribution on $\left[0,1\right]$, and $I_{1}$, $I_{2}$
be independent of $D_{1}$. Recall that
\[
\left(\mu_{f}\left(B_{x,\left\Vert X_{n+1}-x\right\Vert }\right),\mu_{f}\left(B_{x,\left\Vert x-X_{1}\right\Vert }\right)\right)=\left(I_{1},I_{2}\right)\text{ in distribution.}
\]
Then, the rest of the proof is dedicated to showing that, for all
$z$ sufficiently small, (\ref{eq:rewriting prob to be estimated step 2})
is well approximated by
\begin{equation}
\mathbb{P}\left(D_{1}\cdot\max\left\{ I_{1},I{}_{2}\right\} \leq z\right)=z^{2}\mathbb{E}\left[\frac{1}{D_{1}^{2}}\right]\label{eq:equiv prob in Step 2}
\end{equation}
which leads to (\ref{eq:prob to be estimated step 1}). This task
will be carried out by a coupling technique combined with geometric
arguments. \\

\noindent \textbf{Step 3. Introduce the coupling:} In fact, the coupling
technique we will adopt here is identical to the one used in \cite{Devroye2017}.
In \cite{Devroye2017}, the coupling method was used to determine
the limit of the second moment of $n\mu_{f}\left(A_{n}\left(x\right)\right)$
in the case where the Voronoi cells contain a fixed nucleus $x$;
we adapt that method to the setting where the Voronoi cells contain
a fixed point $x$, and study the limit of the first moment of $n\mu_{f}\left(L_{n}\left(x\right)\right)$.
We take $W$ (as in the definition of $D_{1}$) to be $\mathbb{I}_{\left\{ \left\Vert x-X_{1}\right\Vert \leq\left\Vert x-X_{n+1}\right\Vert \right\} }$,
which is obviously a Bernoulli$\left(\frac{1}{2}\right)$ random variable.
Define $Y_{1}$ and $Y_{2}$ to be the reordering of $X_{1}$ and
$X_{n+1}$ such that $\left\Vert x-Y_{1}\right\Vert \leq\left\Vert x-Y_{2}\right\Vert $.
Clearly, $W$ is independent of $Y_{2}$. Conditioning on $Y_{2}$,
let $V_{1}$ be a random variable that has the uniform distribution
on $B_{x,\left\Vert x-Y_{2}\right\Vert }$ and such that $V_{1}$
is maximally coupled with $Y_{1}$ in the sense that 
\[
\mathbb{P}\left(Y_{1}\neq V_{1}|Y_{2}\right)=\frac{1}{2}\int_{B_{x,\left\Vert x-Y_{2}\right\Vert }}\left|\frac{f\left(u\right)}{\mu_{f}\left(B_{x,\left\Vert x-Y_{2}\right\Vert }\right)}-\frac{1}{\lambda\left(B_{x,\left\Vert x-Y_{2}\right\Vert }\right)}\right|du.
\]
Additionally we define 
\[
\left(V,V^{\prime}\right):=\begin{cases}
\left(V_{1},Y_{2}\right), & \text{if }W=1,\\
\left(Y_{2},V_{1}\right), & \text{if }W=0.
\end{cases}
\]
We would like to argue that $\left(V,V^{\prime}\right)$ approximates
$\left(X_{1},X_{n+1}\right)$ well. To this end, we will again invoke
Lemma \ref{lem:Appendix Lebesgue density theorem} in the Appendix.
Namely, let $\epsilon>0$ be arbitrary and $\delta$ be as in Lemma
\ref{lem:Appendix Lebesgue density theorem}. By the choice of $\left(V,V^{\prime}\right)$
and $V_{1}$, we have that
\begin{equation}
\begin{split} & \sup_{\eta\in\left[0,\delta\right]}\mathbb{P}\left(\left.\left(V,V^{\prime}\right)\neq\left(X_{1},X_{n+1}\right)\right|\left\Vert x-Y_{2}\right\Vert =\eta\right)\\
= & \sup_{\eta\in\left[0,\delta\right]}\mathbb{P}\left(\left.Y_{1}\neq V_{1}\right|\left\Vert x-Y_{2}\right\Vert =\eta\right)\\
= & \frac{1}{2}\sup_{\eta\in\left[0,\delta\right]}\int_{B_{x,\eta}}\left|\frac{f\left(u\right)}{\mu_{f}\left(B_{x,\eta}\right)}-\frac{1}{\lambda\left(B_{x,\eta}\right)}\right|du\leq\epsilon.
\end{split}
\label{eq:estimate of P((V,V') neq (X1,Xn+1)}
\end{equation}

On the other hand, with the new notations, we see that the first factor
involved in the random variable in (\ref{eq:rewriting prob to be estimated step 2})
can be written as 
\[
\begin{split}\frac{\mu_{f}\left(B_{X_{n+1},\left\Vert X_{n+1}-X_{1}\right\Vert }\cup B_{x,\left\Vert x-X_{1}\right\Vert }\right)}{\max\left\{ \mu_{f}\left(B_{x,\left\Vert X_{n+1}-x\right\Vert }\right),\mu_{f}\left(B_{x,\left\Vert x-X_{1}\right\Vert }\right)\right\} } & =\frac{\mu_{f}\left(B_{X_{n+1},\left\Vert X_{n+1}-X_{1}\right\Vert }\cup B_{x,\left\Vert x-X_{1}\right\Vert }\right)}{\mu_{f}\left(B_{x,\left\Vert x-Y_{2}\right\Vert }\right)}.\end{split}
\]
Instead of treating the ratio of the $\mu_{f}-$measures of the two
sets involved in the right hand side above, we first look at the corresponding
ratio of replacing ``$\mu_{f}$'' by ``$\lambda$'', i.e., 
\begin{equation}
\frac{\lambda\left(B_{X_{n+1},\left\Vert X_{n+1}-X_{1}\right\Vert }\cup B_{x,\left\Vert x-X_{1}\right\Vert }\right)}{\lambda\left(B_{x,\left\Vert x-Y_{2}\right\Vert }\right)}.\label{eq:lambda measure ratio}
\end{equation}
Assuming $\left(V,V^{\prime}\right)=\left(X_{1},X_{n+1}\right)$,
we have that 
\[
\left(X_{1},X_{n+1}\right)=\begin{cases}
\left(V_{1},Y_{2}\right) & \text{if }W=1,\\
\left(Y_{2},V_{1}\right) & \text{if }W=0,
\end{cases}
\]
and hence the ratio concerned in (\ref{eq:lambda measure ratio})
becomes 
\begin{equation}
\tilde{D}:=\frac{\lambda\left(B_{V_{1},\left\Vert V_{1}-Y_{2}\right\Vert }\cup B_{x,\left\Vert x-Y_{2}\right\Vert }\right)}{\lambda\left(B_{x,\left\Vert x-Y_{2}\right\Vert }\right)}\mathbb{I}_{\left\{ W=0\right\} }+\frac{\lambda\left(B_{Y_{2},\left\Vert Y_{2}-V_{1}\right\Vert }\cup B_{x,\left\Vert x-V_{1}\right\Vert }\right)}{\lambda\left(B_{x,\left\Vert x-Y_{2}\right\Vert }\right)}\mathbb{I}_{\left\{ W=1\right\} }.\label{eq:expression of D_1 Step 3}
\end{equation}

The next fact we will establish is that, when conditioning on $Y_{2}$,
$\tilde{D}$ has the same distribution as $D_{1}$, as defined in
the statement of Theorem \ref{thm:limiting distribution of n*L_n(x), second result in section 3}.
To see this, define $Z:=U_{1}\left\Vert x-Y_{2}\right\Vert +x$ where
$U_{1}$ is a random variable with the uniform distribution on $B_{0,1}$
and independent of $Y_{2}$, so that conditioning on $Y_{2}$, $Z$
has the uniform distribution on $B_{x,\left\Vert x-Y_{2}\right\Vert }$
which is the same as the distribution of $V_{1}$. Recall that $\bar{1}:=\left(1,0,\cdots,0\right)$.
Then, one observes that given $Y_{2}$, 
\[
\begin{split}\tilde{D}\;\stackrel{\text{in dist.}}{=} & \frac{\lambda\left(B_{Z,\left\Vert Y_{2}-Z\right\Vert }\cup B_{x,\left\Vert x-Y_{2}\right\Vert }\right)}{\lambda\left(B_{x,\left\Vert x-Y_{2}\right\Vert }\right)}\mathbb{I}_{\left\{ W=0\right\} }+\frac{\lambda\left(B_{Y_{2},\left\Vert Y_{2}-Z\right\Vert }\cup B_{x,\left\Vert Z-x\right\Vert }\right)}{\lambda\left(B_{x,\left\Vert x-Y_{2}\right\Vert }\right)}\mathbb{I}_{\left\{ W=1\right\} }\\
=\quad & \frac{\lambda\left(B_{U_{1},\left\Vert \bar{1}-U_{1}\right\Vert }\cup B_{0,1}\right)}{\lambda\left(B_{0,1}\right)}\mathbb{I}_{\left\{ W=0\right\} }+\frac{\lambda\left(B_{\bar{1},\left\Vert \bar{1}-U_{1}\right\Vert }\cup B_{0,\left\Vert U_{1}\right\Vert }\right)}{\lambda\left(B_{0,1}\right)}\mathbb{I}_{\left\{ W=1\right\} }\\
=\quad & D_{1},
\end{split}
\]
where the sets concerned in the second line are just the shifted,
re-scaled and rotated versions of those in the first line. Moreover,
it is also clear from this derivation that given $Y_{2}$, the distribution
of $\tilde{D}$ as defined in (\ref{eq:expression of D_1 Step 3})
does not depend on the specific value of $Y_{2}$. The following Figure
3.2 illustrates the sets concerned in the definition of $\tilde{D}$
and $D_{1}$ in the case when $d=2$. 

\begin{figure}[H]
\includegraphics[scale=0.45]{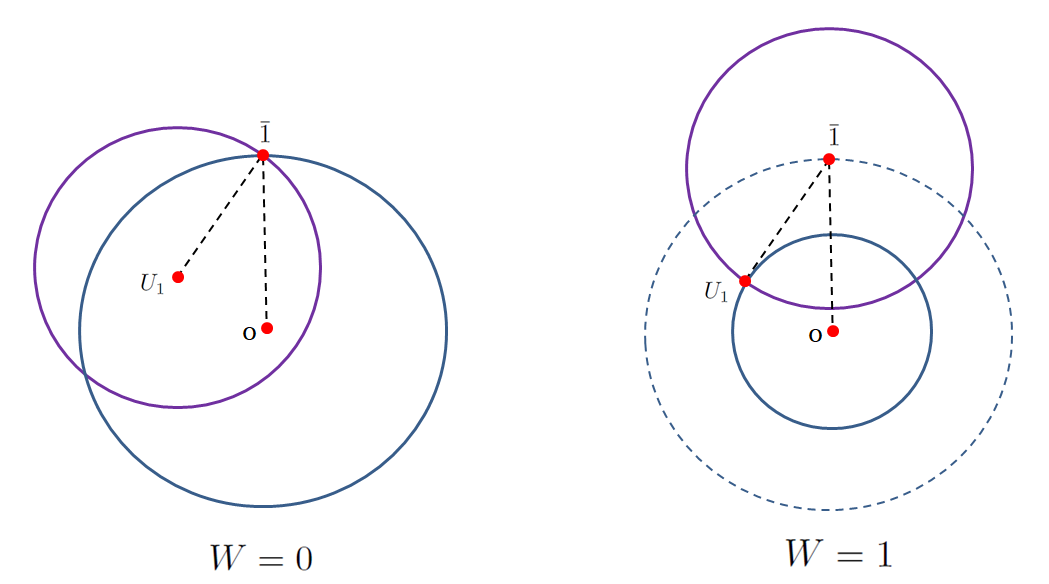}

\caption{The sets concerned in $D_{1}$ (and equivalently $\tilde{D}$) when
$W=0$ or $W=1$ in 2D.}

\end{figure}

\noindent \textbf{Step 4. Establish the probability estimate.} By
the generalized Lebesgue Differentiation Theorem, one can make $\delta>0$
even smaller if necessary such that for all balls $B_{p,r}\subseteq B_{x,\delta}$
with $r\geq\frac{\delta}{4}$, 
\[
\left|\frac{\mu_{f}\left(B_{p,r}\right)}{\lambda\left(B_{p,r}\right)}-f\left(x\right)\right|\leq\frac{f\left(x\right)}{2}.
\]
Consider the event 
\[
\left\{ \left\Vert x-X_{1}\right\Vert >\frac{\delta}{4}\right\} \cup\left\{ \left\Vert x-X_{1}\right\Vert \leq\frac{\delta}{4},\left\Vert x-X_{n+1}\right\Vert >\delta\right\} .
\]
We note that if $\left\Vert x-X_{1}\right\Vert >\frac{\delta}{4}$,
then 
\[
\mu_{f}\left(B_{x,\left\Vert x-X_{1}\right\Vert }\right)\geq\mu_{f}\left(B_{x,\frac{\delta}{4}}\right)\geq\frac{f\left(x\right)}{2}\lambda\left(B_{x,\frac{\delta}{4}}\right);
\]
 if $\left\Vert x-X_{1}\right\Vert \leq\frac{\delta}{4}$ and $\left\Vert x-X_{n+1}\right\Vert >\delta$,
then 
\[
\mu_{f}\left(B_{X_{n+1},\left\Vert X_{1}-X_{n+1}\right\Vert }\right)\geq\mu_{f}\left(B_{p^{*},\frac{\delta}{4}}\right)\geq\frac{f(x)}{2}\lambda\left(B_{p^{*},\frac{\delta}{4}}\right)
\]
where $p^{*}$ is the point on the line segment connecting $X_{1}$
and $X_{n+1}$ such that $\left\Vert X_{1}-p^{*}\right\Vert =\frac{\delta}{4}$.
In particular, we conclude that for every $z\in\left(0,\frac{f\left(x\right)}{2}\lambda\left(B_{0,\frac{\delta}{4}}\right)\right)$,
if 
\[
\mu_{f}\left(B_{X_{n+1},\left\Vert X_{n+1}-X_{1}\right\Vert }\cup B_{x,\left\Vert x-X_{1}\right\Vert }\right)\leq z,
\]
then 
\[
\max\left\{ \mu_{f}\left(B_{X_{n+1},\left\Vert X_{n+1}-X_{1}\right\Vert }\right),\mu_{f}\left(B_{x,\left\Vert x-X_{1}\right\Vert }\right)\right\} \leq z
\]
and hence it must be that
\[
\left\Vert x-X_{1}\right\Vert \leq\frac{\delta}{4}\text{ and }\left\Vert x-X_{n+1}\right\Vert \leq\delta,
\]
which implies that $\left\Vert x-Y_{2}\right\Vert \leq\delta$. Thus,
for all $z$ sufficiently small,
\begin{equation}
\begin{split} & \mathbb{P}\left(\mu_{f}\left(B_{X_{n+1},\left\Vert X_{n+1}-X_{1}\right\Vert }\cup B_{x,\left\Vert x-X_{1}\right\Vert }\right)\leq z\right)\\
= & \mathbb{P}\left(\mu_{f}\left(B_{X_{n+1},\left\Vert X_{n+1}-X_{1}\right\Vert }\cup B_{x,\left\Vert x-X_{1}\right\Vert }\right)\leq z,\left\Vert x-Y_{2}\right\Vert \leq\delta,\left(V,V^{\prime}\right)\neq\left(X_{1},X_{n+1}\right)\right)\\
 & \hspace{1cm}+\mathbb{P}\left(\mu_{f}\left(B_{X_{n+1},\left\Vert X_{n+1}-X_{1}\right\Vert }\cup B_{x,\left\Vert x-X_{1}\right\Vert }\right)\leq z,\left\Vert x-Y_{2}\right\Vert \leq\delta,\left(V,V^{\prime}\right)=\left(X_{1},X_{n+1}\right)\right),
\end{split}
\label{eq:prob estimate (I)+(II) in Step 4}
\end{equation}
and we will treat the two terms on the right hand side of (\ref{eq:prob estimate (I)+(II) in Step 4})
separately.

For every bounded Borel set $B\subseteq\mathbb{R}^{d}$, let $B^{*}$
denote the smallest ball centered at $x$ containing $B$. Let $\epsilon>0$
be chosen as in Step 3. By applying the generalized Lebesgue Differentiation
Theorem again, we can make $\delta$ even smaller such that, for all
bounded Borel sets $B\subseteq\mathbb{R}^{d}$ with $\frac{\lambda(B^{*})}{\lambda(B)}\leq6^{d}$
and $\lambda(B)\leq\left(3\delta\right)^{d}$, we have that
\[
\left|\frac{\mu_{f}\left(B\right)}{\lambda\left(B\right)}-f\left(x\right)\right|\leq\epsilon f\left(x\right).
\]
On one hand, for every $x^{\prime}\in B_{X_{n+1},\left\Vert X_{n+1}-X_{1}\right\Vert }\cup B_{x,\left\Vert x-X_{1}\right\Vert }$,
\[
\left\Vert x-x^{\prime}\right\Vert \leq\max\left\{ \left\Vert X_{n+1}-X_{1}\right\Vert +\left\Vert X_{n+1}-x\right\Vert ,\left\Vert x-X_{1}\right\Vert \right\} \leq3\left\Vert x-Y_{2}\right\Vert 
\]
which means that 
\[
\left(B_{X_{n+1},\left\Vert X_{n+1}-X_{1}\right\Vert }\cup B_{x,\left\Vert x-X_{1}\right\Vert }\right)^{*}\subseteq B_{x,3\left\Vert x-Y_{2}\right\Vert }.
\]
On the other hand, it is always true that 
\[
\max\left\{ \left\Vert X_{n+1}-X_{1}\right\Vert ,\left\Vert x-X_{1}\right\Vert \right\} \geq\frac{1}{2}\left\Vert x-Y_{2}\right\Vert 
\]
which follows from the fact that if $\left\Vert x-X_{1}\right\Vert <\frac{1}{2}\left\Vert x-Y_{2}\right\Vert $,
then 
\[
\left\Vert X_{n+1}-X_{1}\right\Vert \geq\left\Vert X_{n+1}-x\right\Vert -\left\Vert x-X_{1}\right\Vert \geq\frac{1}{2}\left\Vert x-Y_{2}\right\Vert ;
\]
this means that $B_{X_{n+1},\left\Vert X_{n+1}-X_{1}\right\Vert }\cup B_{x,\left\Vert x-X_{1}\right\Vert }$
contains at least one ball with radius $\frac{1}{2}\left\Vert x-Y_{2}\right\Vert $.
Therefore, 
\begin{equation}
\lambda\left(B_{0,1}\right)\left(\frac{1}{2}\left\Vert x-Y_{2}\right\Vert \right)^{d}\leq\lambda\left(B_{X_{n+1},\left\Vert X_{n+1}-X_{1}\right\Vert }\cup B_{x,\left\Vert x-X_{1}\right\Vert }\right)\leq\lambda\left(B_{0,1}\right)\left(3\left\Vert x-Y_{2}\right\Vert \right)^{d}\label{eq:bound of leb measure of the union of balls}
\end{equation}
and
\[
\frac{\lambda\left(\left(B_{X_{n+1},\left\Vert X_{n+1}-X_{1}\right\Vert }\cup B_{x,\left\Vert x-X_{1}\right\Vert }\right)^{*}\right)}{\lambda\left(B_{X_{n+1},\left\Vert X_{n+1}-X_{1}\right\Vert }\cup B_{x,\left\Vert x-X_{1}\right\Vert }\right)}\leq\frac{\left(3\left\Vert x-Y_{2}\right\Vert \right)^{d}}{\left(\frac{1}{2}\left\Vert x-Y_{2}\right\Vert \right)^{d}}=6^{d}.
\]
In particular, we get that if $\left\Vert x-Y_{2}\right\Vert \leq\delta$,
then 

\begin{equation}
\frac{1-\epsilon}{1+\epsilon}\leq\frac{\mu_{f}\left(B_{X_{n+1},\left\Vert X_{n+1}-X_{1}\right\Vert }\cup B_{x,\left\Vert x-X_{1}\right\Vert }\right)}{\mu_{f}\left(B_{x,\left\Vert x-Y_{2}\right\Vert }\right)}\cdot\frac{\lambda\left(B_{x,\left\Vert x-Y_{2}\right\Vert }\right)}{\lambda\left(B_{X_{n+1},\left\Vert X_{n+1}-X_{1}\right\Vert }\cup B_{x,\left\Vert x-X_{1}\right\Vert }\right)}\leq\frac{1+\epsilon}{1-\epsilon}.\label{eq:bound on mu ratio in terms of lambda ratio}
\end{equation}

Now we return to (\ref{eq:prob estimate (I)+(II) in Step 4}). Combining
(\ref{eq:bound of leb measure of the union of balls}) and (\ref{eq:bound on mu ratio in terms of lambda ratio}),
we can estimate the first term on the right hand side of (\ref{eq:prob estimate (I)+(II) in Step 4})
as
\[
\begin{split} & \mathbb{P}\left(\mu_{f}\left(B_{X_{n+1},\left\Vert X_{n+1}-X_{1}\right\Vert }\cup B_{x,\left\Vert x-X_{1}\right\Vert }\right)\leq z,\left\Vert x-Y_{2}\right\Vert \leq\delta,\left(V,V^{\prime}\right)\neq\left(X_{1},X_{n+1}\right)\right)\\
\leq & \mathbb{P}\left(\frac{\lambda\left(B_{X_{n+1},\left\Vert X_{n+1}-X_{1}\right\Vert }\cup B_{x,\left\Vert x-X_{1}\right\Vert }\right)}{\lambda\left(B_{x,\left\Vert x-Y_{2}\right\Vert }\right)}\mu_{f}\left(B_{x,\left\Vert x-Y_{2}\right\Vert }\right)\leq\frac{1+\epsilon}{1-\epsilon}z,\left\Vert x-Y_{2}\right\Vert \leq\delta,\left(V,V^{\prime}\right)\neq\left(X_{1},X_{n+1}\right)\right)\\
\leq & \mathbb{P}\left(\frac{\left(\frac{1}{2}\left\Vert x-Y_{2}\right\Vert \right)^{d}}{\left\Vert x-Y_{2}\right\Vert ^{d}}\mu_{f}\left(B_{x,\left\Vert x-Y_{2}\right\Vert }\right)\leq\frac{1+\epsilon}{1-\epsilon}z,\left\Vert x-Y_{2}\right\Vert \leq\delta,\left(V,V^{\prime}\right)\neq\left(X_{1},X_{n+1}\right)\right)\\
= & \mathbb{P}\left(\mu_{f}\left(B_{x,\left\Vert x-Y_{2}\right\Vert }\right)\leq2^{d}\frac{1+\epsilon}{1-\epsilon}z\right)\cdot\mathbb{P}\left(\left\Vert x-Y_{2}\right\Vert \leq\delta\left|\mu_{f}\left(B_{x,\left\Vert x-Y_{2}\right\Vert }\right)\leq2^{d}\frac{1+\epsilon}{1-\epsilon}z\right.\right)\\
 & \hspace{2cm}\cdot\mathbb{P}\left(\left(V,V^{\prime}\right)\neq\left(X_{1},X_{n+1}\right)\left|\left\Vert x-Y_{2}\right\Vert \leq\delta,\mu_{f}\left(B_{x,\left\Vert x-Y_{2}\right\Vert }\right)\leq2^{d}\frac{1+\epsilon}{1-\epsilon}z\right.\right)\\
\leq & \mathbb{P}\left(\mu_{f}\left(B_{x,\left\Vert x-X_{1}\right\Vert }\right)\leq2^{d}\frac{1+\epsilon}{1-\epsilon}z\right)\cdot\mathbb{P}\left(\mu_{f}\left(B_{x,\left\Vert x-X_{n+1}\right\Vert }\right)\leq2^{d}\frac{1+\epsilon}{1-\epsilon}z\right)\\
 & \hspace{4cm}\cdot\sup_{\eta\leq\delta}\mathbb{P}\left(\left.\left(V,V^{\prime}\right)\neq\left(X_{1},X_{n+1}\right)\right|\left\Vert x-Y_{2}\right\Vert \leq\eta\right)\\
\leq & \left(2^{d}\frac{1+\epsilon}{1-\epsilon}\right)^{2}z^{2}\epsilon\text{ by (\ref{eq:estimate of P((V,V') neq (X1,Xn+1)})}.
\end{split}
\]
The last inequality above is due to the fact that 
\[
\mu_{f}\left(B_{x,\left\Vert x-Y_{2}\right\Vert }\right)\leq2^{d}\frac{1+\epsilon}{1-\epsilon}z\Longleftrightarrow\left\Vert x-Y_{2}\right\Vert \leq r^{*}
\]
where 
\[
r^{*}:=\sup\left\{ s\geq0:\int_{B_{x,s}}f\left(y\right)dy\leq2^{d}\frac{1+\epsilon}{1-\epsilon}z\right\} ,
\]
and hence 
\[
\left\{ \left\Vert x-Y_{2}\right\Vert \leq\delta,\mu_{f}\left(B_{x,\left\Vert x-Y_{2}\right\Vert }\right)\leq2^{d}\frac{1+\epsilon}{1-\epsilon}z\right\} =\left\{ \left\Vert x-Y_{2}\right\Vert \leq\min\left\{ \delta,r^{*}\right\} \right\} .
\]

We move on to the second term in (\ref{eq:prob estimate (I)+(II) in Step 4}).
Recalling the observation we made on (\ref{eq:expression of D_1 Step 3})
when $\left(V,V^{\prime}\right)=\left(X_{1},X_{n+1}\right)$ in Step
3, we have that
\begin{equation}
\begin{split} & \mathbb{P}\left(\mu_{f}\left(B_{X_{n+1},\left\Vert X_{n+1}-X_{1}\right\Vert }\cup B_{x,\left\Vert x-X_{1}\right\Vert }\right)\leq z,\left\Vert x-Y_{2}\right\Vert \leq\delta,\left(V,V^{\prime}\right)=\left(X_{1},X_{n+1}\right)\right)\\
\leq & \mathbb{P}\left(\frac{\lambda\left(B_{X_{n+1},\left\Vert X_{n+1}-X_{1}\right\Vert }\cup B_{x,\left\Vert x-X_{1}\right\Vert }\right)}{\lambda\left(B_{x,\left\Vert x-Y_{2}\right\Vert }\right)}\mu_{f}\left(B_{x,\left\Vert x-Y_{2}\right\Vert }\right)\leq\frac{1+\epsilon}{1-\epsilon}z,\left(V,V^{\prime}\right)=\left(X_{1},X_{n+1}\right)\right)\\
\leq & \mathbb{P}\left(\tilde{D}\cdot\mu_{f}\left(B_{x,\left\Vert x-Y_{2}\right\Vert }\right)\leq\frac{1+\epsilon}{1-\epsilon}z\right)\text{ where }\tilde{D}\text{ is as in (\ref{eq:expression of D_1 Step 3})}\\
= & \mathbb{E}\left[\mathbb{P}\left(\left.\tilde{D}\cdot\mu\left(B_{x,\left\Vert x-Y_{2}\right\Vert }\right)\leq\frac{1+\epsilon}{1-\epsilon}z\right|Y_{2}\right)\right]\\
= & \mathbb{P}\left(D_{1}\cdot\max\left\{ I_{1},I_{2}\right\} \leq\frac{1+\epsilon}{1-\epsilon}z\right)\text{ using same notations as in (\ref{eq:equiv prob in Step 2})}\\
= & z^{2}\left(\frac{1+\epsilon}{1-\epsilon}\right)^{2}\mathbb{E}\left[\frac{1}{D_{1}^{2}}\right].
\end{split}
\label{eq:limsup bound term 2}
\end{equation}
Combining all the arguments above, since $\epsilon>0$ is arbitrarily
small, we can conclude that 
\[
\limsup_{z\rightarrow0}z^{-2}\mathbb{P}\left(\mu_{f}\left(B_{X_{n+1},\left\Vert X_{n+1}-X_{1}\right\Vert }\cup B_{x,\left\Vert x-X_{1}\right\Vert }\right)\leq z\right)\leq\mathbb{E}\left[\frac{1}{D_{1}^{2}}\right].
\]

We now derive the lower bound needed to establish (\ref{eq:prob to be estimated step 1}).
When $z$ is sufficiently small, in particular, when $2^{d}z<\mu_{f}\left(B_{x,\delta}\right)$,
if 
\[
\frac{\lambda\left(B_{X_{n+1},\left\Vert X_{n+1}-X_{1}\right\Vert }\cup B_{x,\left\Vert x-X_{1}\right\Vert }\right)}{\lambda\left(B_{x,\left\Vert x-Y_{2}\right\Vert }\right)}\mu_{f}\left(B_{x,\left\Vert x-Y_{2}\right\Vert }\right)\leq\frac{1-\epsilon}{1+\epsilon}z,
\]
then, by (\ref{eq:bound of leb measure of the union of balls}), it
must be that 
\[
\mu_{f}\left(B_{x,\left\Vert x-Y_{2}\right\Vert }\right)\leq2^{d}z<\mu_{f}\left(B_{x,\delta}\right)
\]
which implies that $\left\Vert x-Y_{2}\right\Vert \leq\delta$. Therefore,
following (\ref{eq:prob estimate (I)+(II) in Step 4})-(\ref{eq:limsup bound term 2}),
we can derive the following estimate for $z$ sufficiently small:
\[
\begin{split} & \mathbb{P}\left(\mu_{f}\left(B_{X_{n+1},\left\Vert X_{n+1}-X_{1}\right\Vert }\cup B_{x,\left\Vert x-X_{1}\right\Vert }\right)\leq z\right)\\
\geq & \mathbb{P}\left(\mu_{f}\left(B_{X_{n+1},\left\Vert X_{n+1}-X_{1}\right\Vert }\cup B_{x,\left\Vert x-X_{1}\right\Vert }\right)\leq z,\left\Vert x-Y_{2}\right\Vert \leq\delta,\left(V,V^{\prime}\right)=\left(X_{1},X_{n+1}\right)\right)\\
\geq & \mathbb{P}\left(\frac{\lambda\left(B_{X_{n+1},\left\Vert X_{n+1}-X_{1}\right\Vert }\cup B_{x,\left\Vert x-X_{1}\right\Vert }\right)}{\lambda\left(B_{x,\left\Vert x-Y_{2}\right\Vert }\right)}\mu_{f}\left(B_{x,\left\Vert x-Y_{2}\right\Vert }\right)\leq\frac{1-\epsilon}{1+\epsilon}z,\left\Vert x-Y_{2}\right\Vert \leq\delta,\left(V,V^{\prime}\right)=\left(X_{1},X_{n+1}\right)\right)\\
= & \mathbb{P}\left(\frac{\lambda\left(B_{X_{n+1},\left\Vert X_{n+1}-X_{1}\right\Vert }\cup B_{x,\left\Vert x-X_{1}\right\Vert }\right)}{\lambda\left(B_{x,\left\Vert x-Y_{2}\right\Vert }\right)}\mu_{f}\left(B_{x,\left\Vert x-Y_{2}\right\Vert }\right)\leq\frac{1-\epsilon}{1+\epsilon}z,\left\Vert x-Y_{2}\right\Vert \leq\delta\right)\\
 & \hspace{0.2cm}-\mathbb{P}\left(\frac{\lambda\left(B_{X_{n+1},\left\Vert X_{n+1}-X_{1}\right\Vert }\cup B_{x,\left\Vert x-X_{1}\right\Vert }\right)}{\lambda\left(B_{x,\left\Vert x-Y_{2}\right\Vert }\right)}\mu_{f}\left(B_{x,\left\Vert x-Y_{2}\right\Vert }\right)\leq\frac{1-\epsilon}{1+\epsilon}z,\left\Vert x-Y_{2}\right\Vert \leq\delta,\left(V,V^{\prime}\right)\neq\left(X_{1},X_{n+1}\right)\right)\\
\geq & \mathbb{P}\left(\frac{\lambda\left(B_{X_{n+1},\left\Vert X_{n+1}-X_{1}\right\Vert }\cup B_{x,\left\Vert x-X_{1}\right\Vert }\right)}{\lambda\left(B_{x,\left\Vert x-Y_{2}\right\Vert }\right)}\mu_{f}\left(B_{x,\left\Vert x-Y_{2}\right\Vert }\right)\leq\frac{1-\epsilon}{1+\epsilon}z\right)\\
 & \hspace{4cm}-\mathbb{P}\left(\mu_{f}\left(B_{x,\left\Vert x-Y_{2}\right\Vert }\right)\leq2^{d}\frac{1-\epsilon}{1+\epsilon}z,\left\Vert x-Y_{2}\right\Vert \leq\delta,\left(V,V^{\prime}\right)\neq\left(X_{1},X_{n+1}\right)\right)\\
\geq & \mathbb{P}\left(\tilde{D}\cdot\mu_{f}\left(B_{x,\left\Vert x-Y_{2}\right\Vert }\right)\leq\frac{1-\epsilon}{1+\epsilon}z\right)-\mathbb{P}\left(\mu_{f}\left(B_{x,\left\Vert x-Y_{2}\right\Vert }\right)\leq2^{d}\frac{1-\epsilon}{1+\epsilon}z,\left\Vert x-Y_{2}\right\Vert \leq\delta,\left(V,V^{\prime}\right)\neq\left(X_{1},X_{n+1}\right)\right).
\end{split}
\]
where $\tilde{D}$ is as in (\ref{eq:expression of D_1 Step 3}).
The first term on the right hand side can be treated in exactly the
same way as in (\ref{eq:limsup bound term 2}), and it leads to 
\[
\mathbb{P}\left(\tilde{D}\cdot\mu_{f}\left(B_{x,\left\Vert x-Y_{2}\right\Vert }\right)\leq\frac{1-\epsilon}{1+\epsilon}z\right)=z^{2}\left(\frac{1-\epsilon}{1+\epsilon}\right)^{2}\mathbb{E}\left[\frac{1}{D_{1}^{2}}\right];
\]
the second term (without the ``$-$'' sign) can be bounded from
above by 
\[
\begin{split} & \mathbb{P}\left(\mu_{f}\left(B_{x,\left\Vert x-Y_{2}\right\Vert }\right)\leq2^{d}\frac{1-\epsilon}{1+\epsilon}z,\left\Vert x-Y_{2}\right\Vert \leq\delta,\left(V,V^{\prime}\right)\neq\left(X_{1},X_{n+1}\right)\right)\\
\leq & \mathbb{P}\left(\left(V,V^{\prime}\right)\neq\left(X_{1},X_{n+1}\right)\left|\left\Vert x-Y_{2}\right\Vert \leq\delta,\mu_{f}\left(B_{x,\left\Vert x-Y_{2}\right\Vert }\right)\leq2^{d}z\right.\right)\mathbb{P}\left(\mu_{f}\left(B_{x,\left\Vert x-Y_{2}\right\Vert }\right)\leq2^{d}z\right)\\
\leq & 2^{2d}z^{2}\epsilon\text{ by (\ref{eq:estimate of P((V,V') neq (X1,Xn+1)}).}
\end{split}
\]
Since $\epsilon>0$ is arbitrarily small, we have proven that 
\[
\liminf_{z\rightarrow0}z^{-2}\mathbb{P}\left(\mu_{f}\left(B_{X_{n+1},\left\Vert X_{n+1}-X_{1}\right\Vert }\cup B_{x,\left\Vert x-X_{1}\right\Vert }\right)\leq z\right)\geq\mathbb{E}\left[\frac{1}{D_{1}^{2}}\right].
\]
(\ref{eq:prob to be estimated step 1}) follows from here, and we
have proven (\ref{eq:theorem result for 1st moment}).\\

\noindent \textbf{Step 5: Determination of the limiting distribution:}
The proof for higher moments is completely similar. As in the case
$k=1$, $W=0$ will correspond to $X_{1}$ being the farthest point
from $x$ among-st $X_{1},X_{n+1},\cdots,X_{n+k}$ and $W=1$ will
correspond to $X_{1}$ being closer to $x$ than some other point
amongst $X_{n+1},\cdots,X_{n+k}$. Here, $X_{n+1},\cdots,X_{n+k}$
are new i.i.d. random variables that are independent of $X_{1},\cdots,X_{n}$
and have probability density function $f$. The role of $X_{n+1},\cdots,X_{n+k}$
in the proof is analogous to that of $X_{n+1}$ in the arguments above. 

Finally, we notice that for every $k\geq1$, $\frac{1}{D_{k}}\leq2^{d}$.
Therefore, for every $k\geq1$, 
\[
\mathbb{E}\left[\frac{\left(k+1\right)!}{D_{k}^{k+1}}\right]\leq2^{d(k+1)}\left(k+1\right)!
\]
and it follows that 
\[
\sum_{k=1}^{\infty}\left(\mathbb{E}\left[\frac{\left(k+1\right)!}{D_{k}^{k+1}}\right]\right)^{-\frac{1}{2k}}\geq\sum_{k=1}^{\infty}\left(2^{d(k+1)}\left(k+1\right)!\right)^{-\frac{1}{2k}}\geq2^{-d}\sum_{k=1}^{\infty}\left(\left(k+1\right)!\right)^{-\frac{1}{2k}}=\infty.
\]
Therefore, by Carleman's condition, these moments determine a unique
limiting distribution and the distribution of $n\mu_{f}\left(L_{n}\left(x\right)\right)$
weakly converges to this limit. The proof of Theorem \ref{thm:limiting distribution of n*L_n(x), second result in section 3}
is completed.
\end{proof}
While this result of the measure of $L_{n}\left(x\right)$ under $\mu_{f}$
is informative, in general, without \emph{a priori} knowledge of $f$,
only the Lebesgue measure of $L_{n}\left(x\right)$ will be observed
in the Voronoi diagram. With Theorems \ref{thm:diameter control on L_n(x), first result in section 3}
and Theorem \ref{thm:limiting distribution of n*L_n(x), second result in section 3}
in hand, we are now able to determine the relation between $f\left(x\right)$
and the asymptotics of the Lebesgue measure of $L_{n}\left(x\right)$,
which constitutes Theorem \ref{thm: convergence in distirbution for leb of L_n(x), 3rd result in Section 3}
stated below. \\

\noindent \textbf{Theorem \ref{thm: convergence in distirbution for leb of L_n(x), 3rd result in Section 3}.
}\emph{Let $f$ be any probability density function on $\mathbb{R}^{d}$,
$x$ be a Lebesgue point of $f$ such that $f(x)>0$, and $Z$ be
a random variable with the distribution $\mathscr{D}$ defined in
Theorem \ref{thm:limiting distribution of n*L_n(x), second result in section 3}.
Then, 
\[
nf\left(x\right)\lambda\left(L_{n}\left(x\right)\right)\rightarrow Z\text{ in distribution}.
\]
}
\begin{proof}
By Theorem \ref{thm:limiting distribution of n*L_n(x), second result in section 3},
we know that $n\mu_{f}\left(L_{n}\left(x\right)\right)\rightarrow Z$
in distribution. Therefore, by Slutsky's theorem, it is enough to
show that 
\[
\frac{\lambda\left(L_{n}\left(x\right)\right)}{\mu_{f}\left(L_{n}\left(x\right)\right)}\rightarrow\frac{1}{f\left(x\right)}\text{ in probability.}
\]
Let $\epsilon>0$ be arbitrary. By the Lebesgue Differentiation Theorem,
for every integer $c\geq1$, there exists $R_{c}>0$ such that if
$B\subseteq\mathbb{R}^{d}$ is a bounded Borel set with $\frac{\lambda\left(B^{*}\right)}{\lambda\left(B\right)}\leq c$,
where $B^{*}$ is the smallest ball centered at $x$ containing $B$,
and $\lambda\left(B\right)\leq R_{c}$, then 
\[
\left|\frac{\lambda\left(B\right)}{\mu_{f}\left(B\right)}-\frac{1}{f\left(x\right)}\right|<\epsilon.
\]
Without loss of generality, we can assume that $R_{c}$ is non-increasing
in $c$. Therefore, if 
\[
\left|\frac{\lambda\left(L_{n}\left(x\right)\right)}{\mu_{f}\left(L_{n}\left(x\right)\right)}-\frac{1}{f\left(x\right)}\right|\geq\epsilon,
\]
then it must be that, for every $c\geq1$, either $\frac{\lambda\left(L_{n}^{*}\left(x\right)\right)}{\lambda\left(L_{n}\left(x\right)\right)}>c$
or $\lambda\left(L_{n}\left(x\right)\right)>R_{c}$, and hence 
\[
\begin{split}\mathbb{P}\left(\left|\frac{\lambda\left(L_{n}\left(x\right)\right)}{\mu_{f}\left(L_{n}\left(x\right)\right)}-\frac{1}{f\left(x\right)}\right|\geq\epsilon\right) & \leq\inf_{c\ge1}\left[\mathbb{P}\left(\frac{\lambda\left(L_{n}^{*}\left(x\right)\right)}{\lambda\left(L_{n}\left(x\right)\right)}>c\right)+\mathbb{P}\left(\lambda\left(L_{n}\left(x\right)\right)>R_{c}\right)\right].\end{split}
\]
To show that the probability above goes to 0 as $n\rightarrow\infty$,
it is sufficient to show that for an arbitrarily small $\epsilon^{\prime}>0$,
when $n$ is sufficiently large, we can find $c\geq1$ such that 
\begin{equation}
\mathbb{P}\left(\frac{\lambda\left(L_{n}^{*}\left(x\right)\right)}{\lambda\left(L_{n}\left(x\right)\right)}>c\right)\leq\epsilon^{\prime}\label{eq:target prob estimate (1)}
\end{equation}
and 
\begin{equation}
\mathbb{P}\left(\lambda\left(L_{n}\left(x\right)\right)>R_{c}\right)\leq\epsilon^{\prime}.\label{eq:target prob estimate (2)}
\end{equation}

Let us examine (\ref{eq:target prob estimate (1)}) first. Set $d_{n}^{1}\left(x\right):=\left\Vert x-X_{(1)}\right\Vert $
and $d_{n}^{2}\left(x\right):=\left\Vert x-X_{(2)}\right\Vert $,
where for every $i\in\left\{ 1,\cdots,n\right\} $, $X_{(i)}$ denotes
the $i$th nearest neighbour of $x$ amongst $X_{1},\cdots,X_{n}$.
We observe that if $z\in B_{x,d_{n}^{2}\left(x\right)-d_{n}^{1}\left(x\right)}$,
then
\[
\left\Vert z-X_{(1)}\right\Vert \leq\left\Vert z-x\right\Vert +\left\Vert x-X_{(1)}\right\Vert <d_{n}^{2}\left(x\right)\leq\left\Vert z-X_{(i)}\right\Vert ,\;\forall i\geq2.
\]
In particular, we have that $B_{x,d_{n}^{2}\left(x\right)-d_{n}^{1}\left(x\right)}\subseteq L_{n}\left(x\right)$.
Then, we may conclude that
\begin{equation}
\mathbb{P}\left(\frac{\lambda\left(L_{n}^{*}\left(x\right)\right)}{\lambda\left(L_{n}\left(x\right)\right)}>c\right)\leq\mathbb{P}\left(\frac{\lambda\left(B_{x,D_{n}^{L}\left(x\right)}\right)}{\lambda\left(B_{x,d_{n}^{2}\left(x\right)-d_{n}^{1}\left(x\right)}\right)}>c\right)=\mathbb{P}\left(\frac{D_{n}^{L}\left(x\right)}{d_{n}^{2}\left(x\right)-d_{n}^{1}\left(x\right)}>c^{\frac{1}{d}}\right).\label{eq:prob estimate on leb(Ln*)/leb(Ln)}
\end{equation}
Let $l_{1}$ and $l_{2}$ be positive constants that we will specify
shortly and write $c:=\left(\frac{l_{2}}{l_{1}}\right)^{d}$. By Theorem
\ref{thm:diameter control on L_n(x), first result in section 3},
we may choose $l_{2}$ large enough such that for all $n$ sufficiently
large, 
\[
\mathbb{P}\left(D_{n}^{L}\left(x\right)>\frac{l_{2}}{n^{\frac{1}{d}}}\right)\leq\frac{1}{2}\epsilon^{\prime}.
\]
Next, we will show that by taking $l_{1}$ small, we can make $\mathbb{P}\left(d_{n}^{2}\left(x\right)-d_{n}^{1}\left(x\right)<\frac{l_{1}}{n^{\frac{1}{d}}}\right)$
smaller than $\frac{1}{2}\epsilon^{\prime}$. To this end, let $l_{3}$
be a third positive constant. Recall that $\text{Bin}\left(n,p\right)$
is the binomial distribution with parameters $n$ and $p\in\left(0,1\right)$.
We have that when $n$ is sufficiently large,
\[
\begin{split}\mathbb{P}\left(d_{n}^{2}\left(x\right)>\frac{l_{3}}{n^{\frac{1}{d}}}\right) & =\text{Bin}\left(n,\mu_{f}\left(B_{x,\frac{l_{3}}{n^{\frac{1}{d}}}}\right)\right)\left(\left\{ 0,1\right\} \right)\\
 & \leq\exp\left[1-\frac{1}{2}f\left(x\right)\lambda\left(B_{0,1}\right)l_{3}^{d}+\ln\left(\frac{3}{2}f\left(x\right)\lambda\left(B_{0,1}\right)l_{3}^{d}\right)\right]
\end{split}
\]
where we applied Chernoff's bound (Lemma \ref{lem: Appendix Chernoff's-bound}
in the Appendix) and the Lebesgue Differentiation Theorem. By choosing
$l_{3}$ sufficiently large, we have that 
\[
\mathbb{P}\left(d_{n}^{2}\left(x\right)>\frac{l_{3}}{n^{\frac{1}{d}}}\right)<\frac{1}{6}\epsilon^{\prime}\text{ for all }n\text{ sufficiently large}.
\]
On the other hand, by choosing $l_{1}$ small, Chernoff's bound (Lemma
\ref{lem: Appendix Chernoff's-bound} in the Appendix) leads to that,
for all $n$ sufficiently large,
\[
\begin{split}\mathbb{P}\left(d_{n}^{2}\left(x\right)\leq\frac{l_{1}}{n^{\frac{1}{d}}}\right) & =\text{Bin}\left(n,\mu_{f}\left(B_{x,\frac{l_{1}}{n^{\frac{1}{d}}}}\right)\right)\left([2,\infty)\right)\\
 & \leq\exp\left[2-\frac{1}{2}f\left(x\right)\lambda\left(B_{0,1}\right)l_{1}^{d}+2\ln\left(\frac{3}{4}f\left(x\right)\lambda\left(B_{0,1}\right)l_{1}^{d}\right)\right],
\end{split}
\]
which can be made smaller than $\frac{1}{6}\epsilon^{\prime}$ provided
that $l_{1}$ is sufficiently small. Thus, for all $n$ sufficiently
large,
\begin{equation}
\begin{split}\mathbb{P}\left(d_{n}^{2}\left(x\right)-d_{n}^{1}\left(x\right)<\frac{l_{1}}{n^{\frac{1}{d}}}\right)\leq & \mathbb{P}\left(d_{n}^{2}\left(x\right)-d_{n}^{1}\left(x\right)<\frac{l_{1}}{n^{\frac{1}{d}}},\frac{l_{1}}{n^{\frac{1}{d}}}<d_{n}^{2}\left(x\right)\leq\frac{l_{3}}{n^{\frac{1}{d}}}\right)\\
 & \hspace{3cm}+\mathbb{P}\left(d_{n}^{2}\left(x\right)>\frac{l_{3}}{n^{\frac{1}{d}}}\right)+\mathbb{P}\left(d_{n}^{2}\left(x\right)\leq\frac{l_{1}}{n^{\frac{1}{d}}}\right)\\
\leq & \mathbb{P}\left(d_{n}^{2}\left(x\right)-d_{n}^{1}\left(x\right)<\frac{l_{1}}{n^{\frac{1}{d}}},\frac{l_{1}}{n^{\frac{1}{d}}}<d_{n}^{2}\left(x\right)\leq\frac{l_{3}}{n^{\frac{1}{d}}}\right)+\frac{1}{3}\epsilon^{\prime}.
\end{split}
\label{eq:prob estimate on dn^2 - dn^1}
\end{equation}
 Now, observe that $d_{n}^{1}\left(x\right)$ and $d_{n}^{2}\left(x\right)$
are the first and second smallest values amongst the i.i.d. random
variables $\left\Vert X_{1}-x\right\Vert ,\cdots,\left\Vert X_{n}-x\right\Vert $.
Let $H$ be the cumulative distribution function of $\left\Vert X_{1}-x\right\Vert $
and $h$ be its probability density function. We have that
\[
\begin{split} & \mathbb{P}\left(d_{n}^{2}\left(x\right)-d_{n}^{1}\left(x\right)<\frac{l_{1}}{n^{\frac{1}{d}}},\frac{l_{1}}{n^{\frac{1}{d}}}<d_{n}^{2}\left(x\right)\leq\frac{l_{3}}{n^{\frac{1}{d}}}\right)\\
= & \int_{l_{1}n^{-\frac{1}{d}}}^{l_{3}n^{-\frac{1}{d}}}\int_{y_{2}-l_{1}n^{-\frac{1}{d}}}^{y_{2}}n\left(n-1\right)h\left(y_{1}\right)h\left(y_{2}\right)\left(1-H\left(y_{2}\right)\right)^{n-2}dy_{1}dy_{2}\\
= & n\left(n-1\right)\int_{l_{1}n^{-\frac{1}{d}}}^{l_{3}n^{-\frac{1}{d}}}\left[H\left(y_{2}\right)-H\left(y_{2}-\frac{l_{1}}{n^{\frac{1}{d}}}\right)\right]h\left(y_{2}\right)\left(1-H\left(y_{2}\right)\right)^{n-2}dy_{2}.
\end{split}
\]
For every $y_{2}\in\left(\frac{l_{1}}{n^{\frac{1}{d}}},\frac{l_{3}}{n^{\frac{1}{d}}}\right]$
and every $n\geq1$, $\left(B_{x,y_{2}}\backslash B_{x,y_{2}-l_{1}n^{-\frac{1}{d}}}\right)^{*}\subseteq B_{x,y_{2}}$,
\[
\frac{\lambda\left(\left(B_{x,y_{2}}\backslash B_{x,y_{2}-l_{1}n^{-\frac{1}{d}}}\right)^{*}\right)}{\lambda\left(B_{x,y_{2}}\backslash B_{x,y_{2}-l_{1}n^{-\frac{1}{d}}}\right)}\leq\frac{\lambda\left(B_{x,y_{2}}\right)}{\lambda\left(B_{x,y_{2}}\backslash B_{x,y_{2}-l_{1}n^{-\frac{1}{d}}}\right)}\leq\frac{l_{3}^{d}}{l_{3}^{d}-\left(l_{3}-l_{1}\right)^{d}},
\]
and
\[
\lambda\left(B_{x,y_{2}}\backslash B_{x,y_{2}-l_{1}n^{-\frac{1}{d}}}\right)\leq\lambda\left(B_{0,1}\right)\frac{l_{3}^{d}}{n}.
\]
Therefore, using the generalized Lebesgue Differentiation Theorem
again, we have that, when $n$ is sufficiently large,
\[
\begin{split}H\left(y_{2}\right)-H\left(y_{2}-\frac{l_{1}}{n^{\frac{1}{d}}}\right) & =\mathbb{P}\left(y_{2}-\frac{l_{1}}{n^{\frac{1}{d}}}<\left\Vert X_{1}-x\right\Vert \leq y_{2}\right)\\
 & =\mu_{f}\left(B_{x,y_{2}}\backslash B_{x,y_{2}-l_{1}n^{-\frac{1}{d}}}\right)\\
 & \leq\frac{3}{2}f\left(x\right)\lambda\left(B_{x,y_{2}}\backslash B_{x,y_{2}-l_{1}n^{-\frac{1}{d}}}\right)\\
 & =\frac{3}{2}f\left(x\right)\lambda\left(B_{0,1}\right)f\left(x\right)\left[y_{2}^{d}-\left(y_{2}-\frac{l_{1}}{n^{\frac{1}{d}}}\right)^{d}\right]\leq\frac{\frac{3d}{2}f\left(x\right)\lambda\left(B_{0,1}\right)l_{3}^{d-1}l_{1}}{n}.
\end{split}
\]
It follows that, when $n$ is sufficiently large, 
\[
\begin{split} & \mathbb{P}\left(d_{n}^{2}\left(x\right)-d_{n}^{1}\left(x\right)<\frac{l_{1}}{n^{\frac{1}{d}}},\frac{l_{1}}{n^{\frac{1}{d}}}<d_{n}^{2}\left(x\right)\leq\frac{l_{3}}{n^{\frac{1}{d}}}\right)\\
= & n\left(n-1\right)\int_{l_{1}n^{-\frac{1}{d}}}^{l_{3}n^{-\frac{1}{d}}}\left[H\left(y_{2}\right)-H\left(y_{2}-\frac{l_{1}}{n^{\frac{1}{d}}}\right)\right]h\left(y_{2}\right)\left(1-H\left(y_{2}\right)\right)^{n-2}dy_{2}\\
\leq & \frac{3d}{2}f\left(x\right)\lambda\left(B_{0,1}\right)l_{3}^{d-1}l_{1}\left(n-1\right)\int_{l_{1}n^{-\frac{1}{d}}}^{l_{3}n^{-\frac{1}{d}}}h\left(y_{2}\right)\left(1-H\left(y_{2}\right)\right)^{n-2}dy_{2}\\
= & \frac{3d}{2}f\left(x\right)\lambda\left(B_{0,1}\right)l_{3}^{d-1}l_{1}\left[\left(1-H\left(\frac{l_{1}}{n^{\frac{1}{d}}}\right)\right)^{n-1}-\left(1-H\left(\frac{l_{3}}{n^{\frac{1}{d}}}\right)\right)^{n-1}\right]\\
\leq & \frac{3d}{2}f\left(x\right)\lambda\left(B_{0,1}\right)l_{3}^{d-1}l_{1}\leq\frac{1}{6}\epsilon^{\prime}
\end{split}
\]
so long as $l_{1}$ is sufficiently small. Combining with (\ref{eq:prob estimate on dn^2 - dn^1}),
we have obtained that 
\[
\mathbb{P}\left(d_{n}^{2}\left(x\right)-d_{n}^{1}\left(x\right)<\frac{l_{1}}{n^{\frac{1}{d}}}\right)\leq\frac{1}{2}\epsilon^{\prime}
\]
for all $n$ sufficiently large. Returning to (\ref{eq:prob estimate on leb(Ln*)/leb(Ln)}),
we see that
\[
\begin{split}\mathbb{P}\left(\frac{\lambda\left(L_{n}^{*}\left(x\right)\right)}{\lambda\left(L_{n}\left(x\right)\right)}>\left(\frac{l_{2}}{l_{1}}\right)^{d}\right) & =\mathbb{P}\left(\frac{D_{n}^{L}\left(x\right)}{d_{n}^{2}\left(x\right)-d_{n}^{1}\left(x\right)}>\frac{l_{2}}{l_{1}}\right)\\
 & \leq\mathbb{P}\left(D_{n}^{L}\left(x\right)>\frac{l_{2}}{n^{\frac{1}{d}}}\right)+\mathbb{P}\left(d_{n}^{2}\left(x\right)-d_{n}^{1}\left(x\right)<\frac{l_{1}}{n^{\frac{1}{d}}}\right)\\
 & \leq\frac{1}{2}\epsilon^{\prime}+\frac{1}{2}\epsilon^{\prime}=\epsilon^{\prime}
\end{split}
\]
for all $n$ sufficiently large, which confirms (\ref{eq:target prob estimate (1)}).

As for (\ref{eq:target prob estimate (2)}), for the specific choice
of $l_{1}$, $l_{2}$ and $c=\left(\frac{l_{2}}{l_{1}}\right)^{d}$,
$R_{c}$ is a positive constant. We apply Theorem \ref{thm:diameter control on L_n(x), first result in section 3}
again to see that 
\[
\begin{split}\mathbb{P}\left(\lambda\left(L_{n}\left(x\right)\right)>R_{c}\right) & \le\mathbb{P}\left(\lambda\left(B_{x,D_{n}^{L}\left(x\right)}\right)>R_{c}\right)=\mathbb{P}\left(D_{n}^{L}\left(x\right)>\left(\frac{R_{c}}{\lambda\left(B_{0,1}\right)}\right)^{\frac{1}{d}}\right)\leq\epsilon^{\prime}\end{split}
\]
by making $n$ even larger if necessary. The proof of Theorem \ref{thm: convergence in distirbution for leb of L_n(x), 3rd result in Section 3}
is thus completed. 
\end{proof}

\section{Asymptotic Independence of Measures of Disjoint Voronoi Cells}

\noindent We will conclude this work by showing that for large $n$,
the ``configurations'' of disjoint regions of the Voronoi diagram
are ``almost'' independent of one another. We state two versions
of this result, one for each of the settings considered above. Since
the proofs of these two theorems are identical we only provide the
proof of Theorem \ref{thm:asymptotic independence result of Section 4}.

In the setting where the Voronoi cells are assumed to contain fixed
points, we have the following result.\\

\noindent \textbf{Theorem \ref{thm:asymptotic independence result of Section 4}.}
\emph{Let $k\geq2$ be an integer and $Z_{1},\cdots,Z_{k}$ be i.i.d.
random variables with the distribution $\mathscr{D}$ defined in Theorem
\ref{thm:limiting distribution of n*L_n(x), second result in section 3}.
Assume that $f$ is a probability density function on $\mathbb{R}^{d}$
and $x_{1},\cdots,x_{k}$ are $k$ distinct Lebesgue points of $f$
such that $f\left(x_{1}\right),\cdots,f\left(x_{k}\right)$ are all
positive. Then, 
\[
\left(n\mu_{f}\left(L_{n}\left(x_{1}\right)\right),\cdots,n\mu_{f}\left(L_{n}\left(x_{k}\right)\right)\right)\rightarrow\left(Z_{1},\cdots,Z_{k}\right)\text{ in distribution}.
\]
}

In the setting where the Voronoi cells are assumed to have fixed nuclei,
we also have ``asymptotic'' independence of the measures of the
cells. \\
\begin{thm}
Let $k\geq2$ be an integer and $Z_{1}^{\prime},\cdots,Z_{k}^{\prime}$
be i.i.d. random variables following the limiting distribution defined
in Theorem 1 of \cite{Devroye2017}. \emph{Assume that $f$ is a probability
density function on $\mathbb{R}^{d}$ and $x_{1},\cdots,x_{k}$ are
$k$ distinct Lebesgue points of $f$ such that $f\left(x_{1}\right),\cdots,f\left(x_{k}\right)$
are all positive.} Let $\mu_{f}\left(A_{n}^{\prime}\left(x_{1}\right)\right),\cdots,\mu_{f}\left(A_{n}^{\prime}\left(x_{k}\right)\right)$
be the Voronoi cells with nuclei $x_{1},\cdots,x_{k}$, respectively,
in the Voronoi diagram generated by $\left\{ x_{1},\cdots,x_{k},X_{1},\cdots,X_{n}\right\} $.
Then, 
\[
\left(n\mu_{f}\left(A_{n}^{\prime}\left(x_{1}\right)\right),\cdots,n\mu_{f}\left(A_{n}^{\prime}\left(x_{k}\right)\right)\right)\rightarrow\left(Z_{1}^{\prime},\cdots,Z_{k}^{\prime}\right)\text{, in distribution}.
\]
\end{thm}

\begin{proof}
\textit{(Proof of Theorem \ref{thm:asymptotic independence result of Section 4})}
We will only provide an explicit proof in the case $k=2$. The proof
of the general case is highly similar, but has cumbersome notations.
Let $F_{x_{1},x_{2}}^{n}$ be the joint distribution function of $\left(n\mu_{f}\left(L_{n}\left(x_{1}\right)\right),n\mu_{f}\left(L_{n}\left(x_{2}\right)\right)\right)$
and $F_{x_{1}}^{n}$ and $F_{x_{2}}^{n}$ be the corresponding marginal
distribution functions. Take $F_{Z}$ to be the distribution function
for $Z_{1}$ and recall that by Theorem \ref{thm:limiting distribution of n*L_n(x), second result in section 3}
we have that, for every $z\in\mathbb{R}^{d}$ that is a continuity
point of $F_{Z}$, $F_{x_{1}}^{n}\left(z\right)\rightarrow F_{Z}\left(z\right)$
and $F_{x_{2}}^{n}\left(z\right)\rightarrow F_{Z}\left(z\right)$
as $n\rightarrow\infty$. Thus, it is enough to show that for every
$z_{1},z_{2}$ both continuity points of $F_{Z}$, 
\begin{equation}
\lim_{n\rightarrow\infty}\left|F_{x_{1},x_{2}}^{n}\left(z_{1},z_{2}\right)-F_{x_{1}}^{n}\left(z_{1}\right)F_{x_{2}}^{n}\left(z_{2}\right)\right|=0.\label{eq:target limit of joint DF v.s. marginal DFs}
\end{equation}

Let $\epsilon>0$ be arbitrary. By Theorem \ref{thm:diameter control on L_n(x), first result in section 3},
we may choose $t$ large, such that for all $n$ sufficiently large,
\[
\mathbb{P}\left(D_{n}^{L}\left(x_{1}\right)>\frac{t}{4n^{\frac{1}{d}}}\right)\leq\frac{\epsilon}{16}
\]
 and 
\[
\mathbb{P}\left(D_{n}^{L}\left(x_{2}\right)>\frac{t}{4n^{\frac{1}{d}}}\right)\leq\frac{\epsilon}{16}.
\]
Recall that for every Borel set $B\subseteq\mathbb{R}^{d}$, $N_{\left\{ X_{1},\cdots,X_{n}\right\} }\left(B\right)$
is the number of points among $\left\{ X_{1},\cdots,X_{n}\right\} $
that fall inside $B$. Let $M$ be a large constant depending only
on $t$ whose value will be specified shortly. Then, for all $n$
sufficiently large, 
\[
\begin{split}\mathbb{P}\left(N_{\left\{ X_{1},\cdots,X_{n}\right\} }\left(B_{x_{1},\frac{t}{n^{\frac{1}{d}}}}\right)>M\right) & =\text{Bin}\left(n,\mu_{f}\left(B_{x_{1},\frac{t}{n^{\frac{1}{d}}}}\right)\right)\left(\left(M,\infty\right)\right)\\
 & \leq\exp\left[M-\frac{1}{2}f\left(x_{1}\right)\lambda\left(B_{0,1}\right)t^{d}-M\ln\left(\frac{2M}{\lambda\left(B_{0,1}\right)3f\left(x_{1}\right)t^{d}}\right)\right]\\
 & \leq\frac{1}{16}\epsilon
\end{split}
\]
so long as $M$ is sufficiently large, where we used Chernoff's bound
(Lemma \ref{lem: Appendix Chernoff's-bound} in the Appendix) again.
Similarly, by choosing $M$ large, we may also ensure that
\[
\mathbb{P}\left(N_{\left\{ X_{1},\cdots,X_{n}\right\} }\left(B_{x_{2},\frac{t}{n^{\frac{1}{d}}}}\right)>M\right)\leq\frac{1}{16}\epsilon
\]
for all $n$ sufficiently large. Thus, if for every $n\geq1$, 
\[
E_{n}:=\left\{ N_{\left\{ X_{1},\cdots,X_{n}\right\} }\left(B_{x_{1},\frac{t}{n^{\frac{1}{d}}}}\right)\leq M,N_{\left\{ X_{1},\cdots,X_{n}\right\} }\left(B_{x_{2},\frac{t}{n^{\frac{1}{d}}}}\right)\leq M,D_{n}^{L}\left(x_{1}\right)\leq\frac{t}{4n^{\frac{1}{d}}},D_{n}^{L}\left(x_{2}\right)\leq\frac{t}{4n^{\frac{1}{d}}}\right\} ,
\]
then $\mathbb{P}\left(E_{n}\right)\geq1-\frac{\epsilon}{4}$ when
$n$ is sufficiently large. Further, we may restrict to $n$ large
such that $B_{x_{1},\frac{t}{n^{\frac{1}{d}}}}\cap B_{x_{2},\frac{t}{n^{\frac{1}{d}}}}=\emptyset$
and $n>2M$. 

Now we turn our attention to (\ref{eq:target limit of joint DF v.s. marginal DFs}).
For every $n\geq1$, $z_{1}$, $z_{2}$ two continuity points of $F_{Z}$,
set 
\[
p_{n}:=F_{x_{1},x_{2}}^{n}\left(z_{1},z_{2}\right)-\mathbb{P}\left(\left\{ n\mu_{f}\left(L_{n}\left(x_{1}\right)\right)\leq z_{1},n\mu_{f}\left(L_{n}\left(x_{2}\right)\right)\leq z_{2}\right\} \cap E_{n}\right).
\]
We remark that $p_{n}\leq\frac{\epsilon}{4}$ for all $n$ sufficiently
large. Thus, we can write 
\[
\begin{split} & F_{x_{1},x_{2}}^{n}\left(z_{1},z_{2}\right)\\
= & \mathbb{P}\left(\left\{ n\mu_{f}\left(L_{n}\left(x_{1}\right)\right)\leq z_{1},n\mu_{f}\left(L_{n}\left(x_{2}\right)\right)\leq z_{2}\right\} \cap E_{n}\right)+p_{n}\\
= & \sum_{i=1}^{M}\sum_{j=1}^{M}\mathbb{P}\left(n\mu_{f}\left(L_{n}\left(x_{1}\right)\right)\leq z_{1},n\mu_{f}\left(L_{n}\left(x_{2}\right)\right)\leq z_{2},D_{n}^{L}\left(x_{1}\right)\leq\frac{t}{4n^{\frac{1}{d}}},D_{n}^{L}\left(x_{2}\right)\leq\frac{t}{4n^{\frac{1}{d}}},\right.\\
 & \hspace{4cm}\left.N_{\left\{ X_{1},\cdots,X_{n}\right\} }\left(B_{x_{1},\frac{t}{n^{\frac{1}{d}}}}\right)=i,N_{\left\{ X_{1},\cdots,X_{n}\right\} }\left(B_{x_{2},\frac{t}{n^{\frac{1}{d}}}}\right)=j\right)+p_{n}\\
= & \sum_{i=1}^{M}\sum_{j=1}^{M}{n \choose i}{n-i \choose j}\mathbb{P}\left(n\mu_{f}\left(L_{n}\left(x_{1}\right)\right)\leq z_{1},n\mu_{f}\left(L_{n}\left(x_{2}\right)\right)\leq z_{2},D_{n}^{L}\left(x_{1}\right)\leq\frac{t}{4n^{\frac{1}{d}}},D_{n}^{L}\left(x_{2}\right)\leq\frac{t}{4n^{\frac{1}{d}}},\right.\\
 & \hspace{1cm}\left.X_{1},\cdots,X_{i}\in B_{x_{1},\frac{t}{n^{\frac{1}{d}}}},X_{i+1},\cdots,X_{i+j}\in B_{x_{2},\frac{t}{n^{\frac{1}{d}}}},X_{i+j+1},\dots,X_{n}\notin B_{x_{1},\frac{t}{n^{\frac{1}{d}}}}\cup B_{x_{2},\frac{t}{n^{\frac{1}{d}}}}\right)+p_{n}.
\end{split}
\]
From here, we proceed similarly as in the proof of Theorem \ref{thm:convergence in distribution main result in Section 2}.
Namely, by forcing $D_{n}^{L}\left(x_{1}\right)\leq\frac{t}{4n^{\frac{1}{d}}}$
(respectively $D_{n}^{L}\left(x_{2}\right)\leq\frac{t}{4n^{\frac{1}{d}}}$)
and Lemma \ref{lem:Appendix range of nuclei relevant to configuration of cell }
in the Appendix, any point outside of the ball $B_{x_{1},\frac{t}{n^{\frac{1}{d}}}}$
(respectively $B_{x_{2},\frac{t}{n^{\frac{1}{d}}}}$) cannot ``contribute''
to the configuration of $L_{n}\left(x_{1}\right)$ (respectively $L_{n}\left(x_{2}\right)$).
Thus, for $q=1,2$, $1\leq k\leq M$ and $1\leq i_{1}<\cdots<i_{k}\leq n$,
if we set
\[
V_{i_{1},\cdots,i_{k}}^{x_{q}}:=\left\{ B_{x_{q},\frac{t}{n^{\frac{1}{d}}}}\cap\left\{ X_{1},\cdots,X_{n}\right\} =\left\{ X_{i_{1}},\cdots,X_{i_{k}}\right\} ,D_{n}^{L}\left(x_{q}\right)\leq\frac{t}{4n^{\frac{1}{d}}},n\mu_{f}\left(L_{n}\left(x_{q}\right)\right)\leq z_{q}\right\} 
\]
and let $L_{i_{1},\cdots,i_{k}}\left(x_{q}\right)$ be the cell containing
$x_{q}$ in the Voronoi diagram generated only by $\left\{ X_{i_{1}},\cdots,X_{i_{k}}\right\} $
and $D_{i_{1},\cdots,i_{k}}^{L}\left(x_{q}\right)$ be the diameter
of $L_{i_{1},\cdots,i_{k}}\left(x_{q}\right)$, then, within $V_{i_{1},\cdots,i_{k}}^{x_{q}}$,
$L_{n}\left(x_{q}\right)=L_{i_{1},\cdots,i_{k}}\left(x_{q}\right)$
and it is easy to see that $V_{i_{1},\cdots,i_{k}}^{x_{q}}$ is identical
to 
\[
\left\{ B_{x_{q},\frac{t}{n^{\frac{1}{d}}}}\cap\left\{ X_{1},\cdots,X_{n}\right\} =\left\{ X_{i_{1}},\cdots,X_{i_{k}}\right\} ,\ D_{i_{1},\cdots,i_{k}}^{L}\left(x_{q}\right)\leq\frac{t}{4n^{\frac{1}{d}}}\text{, and }n\mu_{f}\left(L_{i_{1},\cdots,i_{k}}\left(x_{q}\right)\right)\leq z_{q}\right\} .
\]
Therefore, we can rewrite $F_{x_{1},x_{2}}^{n}\left(z_{1},z_{2}\right)$
as 
\[
\begin{split} & F_{x_{1},x_{2}}^{n}\left(z_{1},z_{2}\right)\\
= & \sum_{i=1}^{M}\sum_{j=1}^{M}{n \choose i}{n-i \choose j}\mathbb{P}\left(V_{1,\cdots,i}^{x_{1}}\cap V_{i+1,\cdots,i+j}^{x_{2}}\cap\left\{ X_{i+j+1},\dots,X_{n}\notin B_{x_{1},\frac{t}{n^{\frac{1}{d}}}}\cup B_{x_{2},\frac{t}{n^{\frac{1}{d}}}}\right\} \right)+p_{n}\\
= & \sum_{i=1}^{M}\sum_{j=1}^{M}{n \choose i}{n-i \choose j}\mathbb{P}\left(V_{1,\cdots,i}^{x_{1}}\right)\mathbb{P}\left(V_{i+1,\cdots,i+j}^{x_{2}}\right)\mathbb{P}\left(X_{i+j+1},\dots,X_{n}\notin B_{x_{1},\frac{t}{n^{\frac{1}{d}}}}\cup B_{x_{2},\frac{t}{n^{\frac{1}{d}}}}\right)+p_{n}.
\end{split}
\]
To proceed from here, we will introduce two more notations. For all
$n\geq1$, define 
\[
\alpha_{n}:=\frac{\sum_{i=1}^{M}\sum_{j=1}^{M}{n \choose i}{n-i \choose j}\mathbb{P}\left(V_{1,\cdots,i}^{x_{1}}\right)\mathbb{P}\left(V_{i+1,\cdots,i+j}^{x_{2}}\right)\mathbb{P}\left(X_{i+j+1},\dots,X_{n}\notin B_{x_{1},\frac{t}{n^{\frac{1}{d}}}}\cup B_{x_{2},\frac{t}{n^{\frac{1}{d}}}}\right)}{\sum_{i=1}^{M}\sum_{j=1}^{M}{n \choose i}{n \choose j}\mathbb{P}\left(V_{1,\cdots,i}^{x_{1}}\right)\mathbb{P}\left(V_{i+1,\cdots,i+j}^{x_{2}}\right)\mathbb{P}\left(X_{i+j+1},\dots,X_{n}\notin B_{x_{1},\frac{t}{n^{\frac{1}{d}}}}\cup B_{x_{2},\frac{t}{n^{\frac{1}{d}}}}\right)}
\]
and 
\[
\beta_{n}:=\frac{\mu_{f}\left(B_{x_{1},\frac{t}{n^{\frac{1}{d}}}}\right)+\mu_{f}\left(B_{x_{2},\frac{t}{n^{\frac{1}{d}}}}\right)}{\left(f\left(x_{1}\right)+f\left(x_{2}\right)\right)\lambda\left(B_{0,1}\right)\frac{t^{d}}{n}}.
\]
The Lebesgue Differentiation Theorem immediately leads to $\lim_{n\rightarrow\infty}\beta_{n}=1$.
We claim that $\lim_{n\rightarrow\infty}\alpha_{n}=1$ as well. To
see this, we check that for every $i,j\in\left\{ 1,\cdots,M\right\} $,
\[
\frac{{n-i \choose j}}{{n \choose j}}=\frac{(n-i)(n-i-1)\cdots(n-i-j+1)}{n(n-1)\cdots(n-j+1)}\in\left[\frac{\left(n-2M+1\right)^{M}}{n^{M}},1\right]
\]
and hence 
\[
\frac{\left(n-2M+1\right)^{M}}{n^{M}}\sum_{i=1}^{M}\sum_{j=1}^{M}{n \choose i}{n \choose j}\leq\sum_{i=1}^{M}\sum_{j=1}^{M}{n \choose i}{n-i \choose j}\leq\sum_{i=1}^{M}\sum_{j=1}^{M}{n \choose i}{n \choose j}
\]
which implies that $\alpha_{n}\rightarrow1$ as $n\rightarrow\infty$.
Thus, we will further rewrite $F_{x_{1},x_{2}}^{n}\left(z_{1},z_{2}\right)$
as 
\begin{equation}
\begin{split} & F_{x_{1},x_{2}}^{n}\left(z_{1},z_{2}\right)\\
= & \alpha_{n}\sum_{i=1}^{M}\sum_{j=1}^{M}{n \choose i}{n \choose j}\mathbb{P}\left(V_{1,\cdots,i}^{x_{1}}\right)\mathbb{P}\left(V_{i+1,\cdots,i+j}^{x_{2}}\right)\mathbb{P}\left(X_{i+j+1},\dots,X_{n}\notin B_{x_{1},\frac{t}{n^{\frac{1}{d}}}}\cup B_{x_{2},\frac{t}{n^{\frac{1}{d}}}}\right)+p_{n}\\
= & \alpha_{n}\sum_{i=1}^{M}\sum_{j=1}^{M}{n \choose i}{n \choose j}\mathbb{P}\left(V_{1,\cdots,i}^{x_{1}}\right)\mathbb{P}\left(V_{1,\cdots,j}^{x_{2}}\right)\left[1-\beta_{n}\left(f\left(x_{1}\right)+f\left(x_{2}\right)\right)\lambda\left(B_{0,1}\right)\frac{t^{d}}{n}\right]^{n-i-j}+p_{n}.
\end{split}
\label{eq:joint DF after rewriting}
\end{equation}

We now examine the quantity $F_{x_{1}}^{n}\left(z_{1}\right)F_{x_{2}}^{n}\left(z_{2}\right)$.
Proceeding as above, for every $n\geq1$, $q=1,2$, we define
\[
\beta_{n}^{q}:=\frac{\mu\left(B_{x_{q},\frac{t}{n^{\frac{1}{d}}}}\right)}{f\left(x_{q}\right)\lambda\left(B_{0,1}\right)\frac{t^{d}}{n}},\;E_{n}^{q}:=\left\{ N_{\left\{ X_{1},\cdots,X_{n}\right\} }\left(B_{x_{q},\frac{t}{n^{\frac{1}{d}}}}\right)\leq M,D_{n}^{L}\left(x_{q}\right)\leq\frac{t}{4n^{\frac{1}{d}}}\right\} ,
\]
and 
\[
\begin{split}p_{n}^{\prime} & :=F_{x_{1}}^{n}\left(z_{1}\right)F_{x_{2}}^{n}\left(z_{2}\right)-\mathbb{P}\left(\left\{ n\mu_{f}\left(L_{n}\left(x_{1}\right)\right)\leq z_{1}\right\} \cap E_{n}^{1}\right)\cdot\mathbb{P}\left(\left\{ n\mu_{f}\left(L_{n}\left(x_{2}\right)\right)\leq z_{2}\right\} \cap E_{n}^{2}\right)\end{split}
.
\]
Then, the arguments above give that 
\[
\mathbb{P}\left(\left(E_{n}^{q}\right)^{\complement}\right)\leq\frac{\epsilon}{8},\,q=1,2,\text{ and hence }p_{n}^{\prime}\leq\frac{3}{8}\epsilon,
\]
as well as 
\[
\lim_{n\rightarrow\infty}\beta_{n}^{q}=1,\,q=1,2.
\]
So, we have that 
\begin{equation}
\begin{split} & F_{x_{1}}^{n}\left(z_{1}\right)F_{x_{2}}^{n}\left(z_{2}\right)\\
= & p_{n}^{\prime}+\mathbb{P}\left(\left\{ n\mu_{f}\left(L_{n}\left(x_{1}\right)\right)\leq z_{1}\right\} \cap E_{n}^{1}\right)\cdot\mathbb{P}\left(\left\{ n\mu_{f}\left(L_{n}\left(x_{2}\right)\right)\leq z_{2}\right\} \cap E_{n}^{2}\right)\\
= & p_{n}^{\prime}+\sum_{i=1}^{M}\sum_{j=1}^{M}{n \choose i}{n \choose j}\mathbb{P}\left(V_{1,\cdots,i}^{x_{1}}\right)\mathbb{P}\left(V_{1,\cdots,j}^{x_{2}}\right)\\
 & \hspace{3cm}\cdot\left[1-\beta_{n}^{1}f\left(x_{1}\right)\lambda\left(B_{0,1}\right)\frac{t^{d}}{n}\right]^{n-i}\left[1-\beta_{n}^{2}f\left(x_{2}\right)\lambda\left(B_{0,1}\right)\frac{t^{d}}{n}\right]^{n-j}.
\end{split}
\label{eq:marginal DFs after rewriting}
\end{equation}

Taking the difference between (\ref{eq:joint DF after rewriting})
and (\ref{eq:marginal DFs after rewriting}), we see that 
\[
\left|F_{x_{1},x_{2}}^{n}\left(z_{1},z_{2}\right)-F_{x_{1}}^{n}\left(z_{1}\right)F_{x_{2}}^{n}\left(z_{2}\right)\right|\leq\Phi_{1}+\Phi_{2}+\Phi_{3}+p_{n}+p_{n}^{\prime}
\]
where 
\[
\begin{split}\Phi_{1} & :=\left|\alpha_{n}-1\right|\sum_{i=1}^{M}\sum_{j=1}^{M}{n \choose i}{n \choose j}\mathbb{P}\left(V_{1,\cdots,i}^{x_{1}}\right)\mathbb{P}\left(V_{1,\cdots,j}^{x_{2}}\right)\left[1-\beta_{n}\left(f\left(x_{1}\right)+f\left(x_{2}\right)\right)\lambda\left(B_{0,1}\right)\frac{t^{d}}{n}\right]^{n-i-j},\end{split}
\]
\[
\begin{split}\Phi_{2} & :=\sum_{i=1}^{M}\sum_{j=1}^{M}{n \choose i}{n \choose j}\mathbb{P}\left(V_{1,\cdots,i}^{x_{1}}\right)\mathbb{P}\left(V_{1,\cdots,j}^{x_{2}}\right)\\
 & \hspace{3.7cm}\cdot\left|\left[1-\beta_{n}\left(f\left(x_{1}\right)+f\left(x_{2}\right)\right)\lambda\left(B_{0,1}\right)\frac{t^{d}}{n}\right]^{n-i-j}-e^{\left(f\left(x_{1}\right)+f\left(x_{2}\right)\right)\lambda\left(B_{0,1}\right)t^{d}}\right|
\end{split}
\]
and
\[
\begin{split}\Phi_{3} & :=\sum_{i=1}^{M}\sum_{j=1}^{M}{n \choose i}{n \choose j}\mathbb{P}\left(V_{1,\cdots,i}^{x_{1}}\right)\mathbb{P}\left(V_{1,\cdots,j}^{x_{2}}\right)\\
 & \hspace{1cm}\cdot\left|\left[1-\beta_{n}^{1}f\left(x_{1}\right)\lambda\left(B_{0,1}\right)\frac{t^{d}}{n}\right]^{n-i}\left[1-\beta_{n}^{2}f\left(x_{2}\right)\lambda\left(B_{0,1}\right)\frac{t^{d}}{n}\right]^{n-j}-e^{\left(f\left(x_{1}\right)+f\left(x_{2}\right)\right)\lambda\left(B_{0,1}\right)t^{d}}\right|.
\end{split}
\]
Therefore, the only thing left is to show that each of $\Phi_{1}$,
$\Phi_{2}$ and $\Phi_{3}$ is no greater than $\frac{\epsilon}{8}$
for all $n$ sufficiently large. We will bound all three of $\Phi_{1}$,
$\Phi_{2}$ and $\Phi_{3}$ at the same time. Since we already know
that as $n\rightarrow\infty$, $\alpha_{n}$, $\beta_{n}$, $\beta_{n}^{1}$
and $\beta_{n}^{2}$ all tend to 1, all we need to do is to bound
\[
\sum_{i=1}^{M}\sum_{j=1}^{M}{n \choose i}{n \choose j}\mathbb{P}\left(V_{1,\cdots,i}^{x_{1}}\right)\mathbb{P}\left(V_{1,\cdots,j}^{x_{2}}\right)=\left(\sum_{i=1}^{M}{n \choose i}\mathbb{P}\left(V_{1,\cdots,i}^{x_{1}}\right)\right)\left(\sum_{i=1}^{M}{n \choose i}\mathbb{P}\left(V_{1,\cdots,i}^{x_{2}}\right)\right)
\]
uniformly in $n$. To this end, we observe that for $q=1,2$,
\[
\begin{split}\limsup_{n\rightarrow\infty}\sum_{i=1}^{M}{n \choose i}\mathbb{P}\left(V_{1,\cdots,i}^{x_{q}}\right) & \leq\limsup_{n\rightarrow\infty}\sum_{i=1}^{M}{n \choose i}\left[\mu_{f}\left(B_{x_{q},\frac{t}{n^{\frac{1}{d}}}}\right)\right]^{i}\\
 & \leq\limsup_{n\rightarrow\infty}\sum_{i=1}^{M}{n \choose i}\left[\frac{3}{2}f\left(x_{q}\right)\lambda\left(B_{0,1}\right)\frac{t^{d}}{n}\right]^{i}\\
 & =\sum_{i=1}^{M}\frac{\left[\frac{3}{2}f\left(x_{q}\right)\lambda\left(B_{0,1}\right)t^{d}f\left(x_{q}\right)\right]^{i}}{i!}.
\end{split}
\]
It follows from there that all of $\Phi_{1}$, $\Phi_{2}$ and $\Phi_{3}$
can be made arbitrarily small, say, smaller than $\frac{\epsilon}{8}$,
so long as $n$ is sufficiently large. Finally we can conclude that
\[
\left|F_{x_{1},x_{2}}^{n}\left(z_{1},z_{2}\right)-F_{x_{1}}^{n}\left(z_{1}\right)F_{x_{2}}^{n}\left(z_{2}\right)\right|\leq\frac{3}{8}\epsilon+p_{n}+p_{n}^{\prime}\leq\epsilon,
\]
which proves (\ref{eq:target limit of joint DF v.s. marginal DFs})
and hence completes the proof of Theorem \ref{thm:asymptotic independence result of Section 4}.
\end{proof}

\section{Appendix}

Following the same notations as in previous sections, assume that
$f$ is a probability density function on $\mathbb{R}^{d}$, $x$
is a Lebesgue point of $f$ with $f\left(x\right)>0$ and $\mu_{f}$
is the probability distribution on $\mathbb{R}^{d}$ with $f$ being
the density function; for each $n\geq1$, let $\left\{ X_{1},\cdots,X_{n}\right\} $
be i.i.d. random variables with distribution $\mu_{f}$, and $\left\{ Y_{1},\cdots,Y_{n},\cdots\right\} $
be a homogeneous Poisson point process on $\mathbb{R}^{d}$ with parameter
$nf\left(x\right)$; finally denote by $A_{n}\left(x\right)$ the
Voronoi cell containing $x$ in the Voronoi diagram generated by $\left\{ x,X_{1},\cdots,X_{n}\right\} $,
$D_{n}^{A}\left(x\right)$ the diameter of $A_{n}\left(x\right)$,
$P_{n}\left(x\right)$ the Voronoi cell containing $x$ in the Voronoi
diagram generated by $\left\{ x,Y_{1},\cdots,Y_{n},\cdots\right\} $,
and $D_{n}^{P}\left(x\right)$ the diameter of $P_{n}\left(x\right)$.
\begin{lem}
\label{lem:Appendix control on diameter of A_n(x) and P_n(x)} Following
the notations introduced above, we have that there exist universal
constants $c_{1},c_{2}>0$ such that $\forall t>0$, 
\[
\limsup\limits _{n\rightarrow\infty}\mathbb{P}\left(D_{n}^{A}\left(x\right)\geq\frac{t}{n^{\frac{1}{d}}}\right)\leq c_{1}e^{-c_{2}f\left(x\right)t^{d}}
\]
and
\[
\limsup\limits _{n\rightarrow\infty}\mathbb{P}\left(D_{n}^{P}\left(x\right)\geq\frac{t}{n^{\frac{1}{d}}}\right)\leq c_{1}e^{-c_{2}f\left(x\right)t^{d}}.
\]
\end{lem}

The first statement is exactly Theorem 5.1 of \cite{Devroye2015}.
The second statement is similar and can be proven by making only minor
adjustments to the proof of the first statement, so we will omit the
details here.

Further, for each $n\geq1$, let $L_{n}\left(x\right)$ be the Voronoi
cell containing $x$ in the Voronoi diagram generated by $\left\{ X_{1},\cdots,X_{n}\right\} $
and $D_{n}^{L}\left(x\right)$ be the diameter of $L_{n}\left(x\right)$.
Then we have the following fact.
\begin{lem}
\label{lem:Appendix range of nuclei relevant to configuration of cell }
Let $t>0$. Suppose that $D_{n}^{A}\left(x\right)\leq\frac{t}{2}$
(alternatively $D_{n}^{P}\left(x\right)\leq\frac{t}{2}$). Then, 
\[
X_{i}\notin B_{x,t}\text{ and }z\in B_{x,\frac{t}{2}}\implies\left\Vert z-x\right\Vert <\left\Vert z-X_{i}\right\Vert .
\]
Similarly, if $D_{n}^{L}\left(x\right)\leq\frac{t}{4}$ and $y$ is
the nucleus of $L_{n}\left(x\right)$, then 
\[
X_{i}\notin B_{x,t}\text{ and }z\in B_{x,\frac{t}{4}}\implies\left\Vert z-y\right\Vert <\left\Vert z-X_{i}\right\Vert .
\]
In particular, we conclude that under the above restrictions on the
diameter, sample points that fall outside of $B_{x,t}$ do not effect
the shape of the cell under consideration. 
\end{lem}

\begin{proof}
Let $z\in B_{x,\frac{t}{2}}$. Then, 
\[
\left\Vert z-x\right\Vert <\frac{t}{2}<\left\Vert X_{i}-x\right\Vert -\left\Vert x-z\right\Vert \leq\left\Vert X_{i}-z\right\Vert .
\]
Similarly, let $z\in B_{x,\frac{t}{4}}$. Then, 
\[
\left\Vert z-y\right\Vert \leq\left\Vert z-x\right\Vert +\left\Vert x-y\right\Vert <\frac{t}{2}<\left\Vert X_{i}-x\right\Vert -\left\Vert z-x\right\Vert \leq\left\Vert X_{i}-z\right\Vert .
\]
\end{proof}
\begin{lem}
\label{lem:Appendix cone arguments on range of cell} Let $\alpha>0$,
$x\in\mathbb{R}^{d}$, and $C\subseteq\mathbb{R}^{d}$ be any cone
of angle $\frac{\pi}{12}$ and with origin at $x$ ($C$ does not
contain $x$), i.e., 
\[
C:=\left\{ y\in\mathbb{R}^{d}\backslash\left\{ x\right\} :\frac{\left(v,y-x\right)_{\mathbb{R}^{d}}}{\left\Vert y-x\right\Vert }\geq\cos\left(\frac{\pi}{24}\right)\right\} \text{, for some }v\in\mathbb{R}^{d}\text{ with }\left\Vert v\right\Vert =1.
\]
Let $R_{1}=\frac{1}{64}\alpha$, $R_{2}=\frac{1+31\cos\left(\frac{\pi}{6}\right)}{64\cos\left(\frac{\pi}{12}\right)}\alpha$,
and $R_{3}=\frac{30}{64}\alpha$. Then, for any $p,y,z\in C$, if
\[
0<\left\Vert y-x\right\Vert <R_{1},\ R_{2}\leq\left\Vert p-x\right\Vert <R_{3},\text{ and }\left\Vert x-z\right\Vert \geq\frac{\alpha}{2},
\]
then we must have 
\[
\left\Vert z-p\right\Vert <\left\Vert z-y\right\Vert .
\]
\end{lem}

\begin{proof}
We claim that it is enough to prove this lemma in the case $d=2$.
First, note that by translation we may assume that $x=0$. Then, let
$y^{\prime}$ be the point on the line segment connecting $0$ and
$z$ such that $\left\Vert y^{\prime}\right\Vert =\left\Vert y\right\Vert $.
It should be clear that $\left\Vert y^{\prime}-z\right\Vert \leq\left\Vert y-z\right\Vert $.
Hence, without loss of generality we may assume that $y=y^{\prime}$.
Use the Gram-Schmidt process to complete $\left\{ p,y\right\} $ to
an orthonormal basis of $\mathbb{R}^{d}$ and consider the problem
in this basis. Since the Euclidean inner product is invariant under
orthogonal transformations, both the Euclidean norm and the cone,
$C$, will be preserved by this transformation. Additionally, by the
use of Gram-Schmidt, we will have that in the new basis $p=\left(p_{1},0,\cdots,0\right)$,
$y=\left(y_{1},y_{2},0,\cdots,0\right)$ and $z=\left(z_{1},z_{2},0,\cdots,0\right)$
for some $p_{1},y_{1},y_{2},z_{1},z_{2}\in\mathbb{R}$. Thus, we see
that we may assume that $d=2$.

\begin{figure}[H]

\includegraphics[scale=0.35]{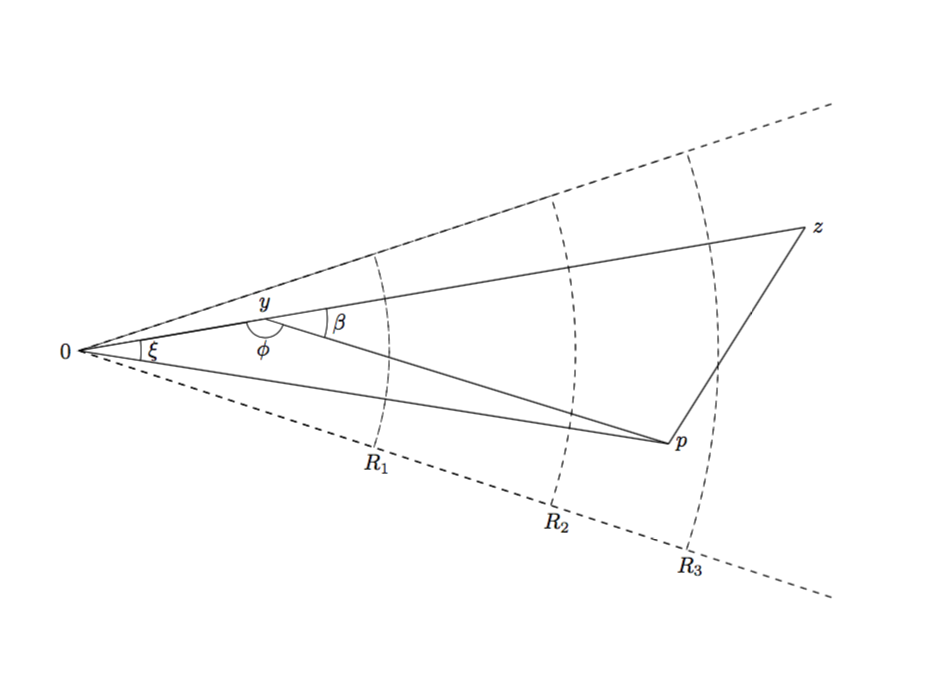}\caption{Diagram of the setting under study in Lemma \ref{lem:Appendix cone arguments on range of cell}. }

\end{figure}

Figure 5.1 outlines the current setting. First we remark that
\[
\begin{split}\left\Vert y-p\right\Vert ^{2} & =\left\Vert x-y\right\Vert ^{2}+\left\Vert x-p\right\Vert ^{2}-2\left\Vert x-y\right\Vert \left\Vert x-p\right\Vert \cos\left(\xi\right)\\
 & \leq\left\Vert x-y\right\Vert ^{2}+\left\Vert x-p\right\Vert ^{2}-2\left\Vert x-y\right\Vert \left\Vert x-p\right\Vert \cos\left(\frac{\pi}{12}\right).
\end{split}
\]
Moreover, 
\[
\begin{split}\left\Vert x-p\right\Vert ^{2} & =\left\Vert x-y\right\Vert ^{2}+\left\Vert y-p\right\Vert ^{2}-2\left\Vert x-y\right\Vert \left\Vert y-p\right\Vert \cos\left(\phi\right)\\
 & \leq2\left\Vert x-y\right\Vert ^{2}+\left\Vert x-p\right\Vert ^{2}-2\left\Vert x-y\right\Vert \left\Vert x-p\right\Vert \cos\left(\frac{\pi}{12}\right)-2\left\Vert x-y\right\Vert \left\Vert y-p\right\Vert \cos\left(\phi\right).
\end{split}
\]
Rewriting the inequality above gives
\[
\cos\left(\phi\right)\leq\frac{\left\Vert x-y\right\Vert -\left\Vert x-p\right\Vert \cos\left(\frac{\pi}{12}\right)}{\left\Vert y-p\right\Vert }\leq\frac{R_{1}-R_{2}\cos\left(\frac{\pi}{12}\right)}{R_{3}+R_{1}}.
\]
Plugging in the values of $R_{1}$, $R_{2}$ and $R_{3}$, we get
that $\phi\geq\frac{5\pi}{6}$ and $\beta=\pi-\phi\leq\frac{\pi}{6}$. 

Now assume the conclusion fails, i.e., $\left\Vert z-y\right\Vert \leq\left\Vert z-p\right\Vert $.
Then, 
\[
\begin{split}\left\Vert z-p\right\Vert ^{2} & =\left\Vert z-y\right\Vert ^{2}+\left\Vert y-p\right\Vert ^{2}-2\left\Vert z-y\right\Vert \left\Vert y-p\right\Vert \cos\left(\beta\right)\\
 & \leq\left\Vert z-y\right\Vert ^{2}+\left\Vert y-p\right\Vert ^{2}-2\left\Vert z-y\right\Vert \left\Vert y-p\right\Vert \cos\left(\frac{\pi}{6}\right)\\
 & \leq\left\Vert z-p\right\Vert ^{2}+\left\Vert y-p\right\Vert ^{2}-2\left\Vert z-y\right\Vert \left\Vert y-p\right\Vert \cos\left(\frac{\pi}{6}\right),
\end{split}
\]
and so
\[
\left\Vert z-y\right\Vert \leq\frac{\left\Vert y-p\right\Vert }{2\cos\left(\frac{\pi}{6}\right)}\leq\frac{R_{1}+R_{3}}{2\cos\left(\frac{\pi}{6}\right)}
\]
which implies that 
\[
\left\Vert z-x\right\Vert \leq\left\Vert z-y\right\Vert +\left\Vert y-x\right\Vert \leq\frac{R_{1}+R_{3}}{2\cos\left(\frac{\pi}{6}\right)}+R_{1}<\frac{\alpha}{2}.
\]
This contradicts the assumption that $\left\Vert z-x\right\Vert \geq\frac{\alpha}{2}$
and thus concludes the proof. 
\end{proof}
\begin{lem}
\label{lem:Appendix Lebesgue density theorem} Let $f$ be a probability
density function on $\mathbb{R}^{d}$ and $x$ be a Lebesgue point
of $f$ such that $f\left(x\right)>0$. Then, $\forall\epsilon\in(0,1)$,
$\exists\delta>0$, such that $\forall\eta\in\left[0,\delta\right]$,
\[
\int_{B_{x,\eta}}\left|\frac{f\left(u\right)}{\mu_{f}\left(B_{x,\eta}\right)}-\frac{1}{\lambda\left(B_{x,\eta}\right)}\right|du\leq\epsilon.
\]
\end{lem}

\begin{proof}
By the generalized Lebesgue Differentiation Theorem (see, e.g., Theorem
20.19 of \cite{Devroye2015}) we may choose $\delta>0$ such that
$\forall\eta\leq\delta$, 
\[
\frac{1}{\lambda\left(B_{x,\eta}\right)}\int_{B_{x,\eta}}\left|f\left(u\right)-f\left(x\right)\right|du\leq\frac{\epsilon f\left(x\right)}{3},
\]
\[
\left|\frac{\mu_{f}\left(B_{x,\eta}\right)}{\lambda\left(B_{x,\eta}\right)}-f\left(x\right)\right|\leq\frac{\epsilon f\left(x\right)}{3},\text{ and }\left|\frac{\lambda\left(B_{x,\eta}\right)}{\mu_{f}\left(B_{x,\eta}\right)}-\frac{1}{f\left(x\right)}|\right|\leq\frac{\epsilon}{3f\left(x\right)}.
\]
Then, 
\[
\begin{split} & \int_{B_{x,\eta}}\left|\frac{f\left(u\right)}{\mu_{f}\left(B_{x,\eta}\right)}-\frac{1}{\lambda\left(B_{x,\eta}\right)}\right|du\\
= & \frac{1}{\mu_{f}\left(B_{x,\eta}\right)}\int_{B_{x,\eta}}\left|f\left(u\right)-\frac{\mu_{f}\left(B_{x,\eta}\right)}{\lambda\left(B_{x,\eta}\right)}\right|du\\
= & \frac{\lambda\left(B_{x,\eta}\right)}{\mu_{f}\left(B_{x,\eta}\right)}\frac{1}{\lambda\left(B_{x,\eta}\right)}\int_{B_{x,\eta}}\left|f\left(u\right)-f\left(x\right)\right|du+\frac{\lambda\left(B_{x,\eta}\right)}{\mu_{f}\left(B_{x,\eta}\right)}\left|f\left(x\right)-\frac{\mu_{f}\left(B_{x,\eta}\right)}{\lambda\left(B_{x,\eta}\right)}\right|\\
\leq & \left(\frac{1}{f\left(x\right)}+\frac{\epsilon}{3f\left(x\right)}\right)\frac{\epsilon f\left(x\right)}{3}+\left(\frac{1}{f\left(x\right)}+\frac{\epsilon}{3f\left(x\right)}\right)\frac{\epsilon f\left(x\right)}{3}\leq\epsilon.
\end{split}
\]
\end{proof}
\begin{lem}
\label{lem: Appendix Chernoff's-bound} (Chernoff's bound \cite{Chernoff1952}).
Let $Z$ be a random variable with the Bin$\left(n,p\right)$ distribution
with parameters $n$ and $p\in(0,1)$. For every $t>0$, set 
\[
\phi\left(t\right):=t-np-t\ln\left(\frac{t}{np}\right).
\]
Then, 
\[
\mathbb{P}\left(Z\geq t\right)\leq e^{\phi\left(t\right)}\text{, for }t\geq np
\]
and 
\[
\mathbb{P}\left(Z\leq t\right)\leq e^{\phi\left(t\right)}\text{, for }0<t\leq np.
\]
\end{lem}

This is the version of Chernoff's bound that will be adopted in this
article. Proof is omitted.

\newpage{}

\bibliographystyle{plain}
\bibliography{VoronoiPaperV2}

\end{document}